\documentclass[12pt,twoside,reqno]{amsart}
\usepackage{amsmath}
\usepackage{amsfonts}
\usepackage{amssymb}
\usepackage{color}
\usepackage{mathrsfs}
\usepackage{cite}
\usepackage{geometry}
\usepackage{marginnote}
\usepackage{todonotes}
\allowdisplaybreaks
\newtheorem{theorem}{Theorem}[section]
\newtheorem{corollary}[theorem]{Corollary}
\newtheorem{lemma}[theorem]{Lemma}
\newtheorem{proposition}[theorem]{Proposition}

\newtheorem{remark}[theorem]{Remark}
\numberwithin{equation}{section}
\begin{document}
\title{Weak solutions to the collision-induced breakage equation with dominating coagulation} 
\thanks{Partially supported by the Indo-French Centre for Applied Mathematics (IFCAM) within the project \textsl{Collision-induced fragmentation and coagulation: dynamics and numerics}}
\author{Ankik Kumar Giri}
\address{Department of Mathematics, Indian Institute of Technology Roorkee\\ Roorkee--247667, Uttarakhand, India}
\email{ankik.giri@ma.iitr.ac.in/ankik.math@gmail.com}

\author{Philippe Lauren\c{c}ot}
\address{Institut de Math\'ematiques de Toulouse, UMR~5219, Universit\'e de Toulouse, CNRS \\ F--31062 Toulouse Cedex 9, France}
\email{laurenco@math.univ-toulouse.fr}


\date{\today}

\begin{abstract}
 Existence and uniqueness of weak solutions to the collision-induced breakage and coagulation equation are shown when coagulation is the dominant mechanism for small volumes. The collision kernel may feature a stronger singularity for small volumes than the ones considered in previous contributions. In addition, when the collision kernel is locally bounded, the class of fragment daughter distribution functions included in the analysis is broader. Mass-conserving solutions are also constructed when the collision kernel grows at most linearly at infinity and are proved to be unique for initial conditions decaying sufficiently fast at infinity. The existence proofs relies on a weak compactness approach in $L^1$.
\end{abstract}

\maketitle

%
%
\pagestyle{myheadings}
\markboth{\sc{A.K. Giri \& Ph. Lauren\c cot}}{\sc{Collision-induced breakage equation with dominating coagulation}}

\section{Introduction}\label{sec1}

Coagulation-fragmentation processes typically occur in the dynamics of particle growth and describe how particles can combine to form larger ones or split into smaller ones. Particle growth models are met in a wide range of contexts, including astrophysics, biology, chemistry, atmospheric science, aerosol science, and population dynamics, to name but a few. Assuming that each particle is completely characterized by a single size variable, such as its volume, a commonly used mathematical model describing coagulation and fragmentation events is known as the classical coagulation-fragmentation equation (CFE), see \cite{Melz57}. While it is well known that coagulation is a nonlinear process, particle breakup can be classified into two categories:  \emph{linear or spontaneous breakage} usually takes place spontaneously or is due to external forces and does not involve interactions between particles in the system under study. \emph{Nonlinear fragmentation}, also known as \emph{collision-induced breakage}, results from collisions between particles in the system \cite{Gili78, LiGi76, Safr72}. It is thus a genuinely nonlinear mechanism, in contrast to spontaneous fragmentation which is linear. Another fundamental difference is that spontaneous breakup only produces smaller daughter particles in general, while collision-induced breakage may allow some transfer of matter between the colliding particles, and thereby produce daughter particles with a size larger than the respective sizes of the parent particles. Within the framework of the classical CFE, it is mainly linear fragmentation which has been taken into account and studied in the mathematical literature since the pioneering works \cite{BaCa90, Fili61, McLe62c, Melz57, Spou84, Stew89, Whit80}, see \cite{BLL19,Bert06,Dubo94b} and the references therein. Only a few mathematical papers include collision-induced breakage \cite{BaGi20, BaGia, BaGib, LaWr01, Walk02, Wil82}, though several contributions are found in the physics literature, with a particular emphasis on dynamical predictions, formal asymptotics, numerical simulations, and special solutions, see \cite{BKBSHSS15, Jaco11, KoKa00, KoKa06, KrBN03, Safr72, Sriv78, Sriv82, VVF06}.
	
 This article is devoted to the well-posedness of the continuous version of the nonlinear \emph{collision-induced breakage and coagulation equation}, which describes the time evolution of the particle size distribution $f=f(t,x)\ge 0$ of particles of volume $x \in (0, \infty)$ at time $t \ge 0$ and reads
\begin{subequations}\label{cecb}
\begin{align}
\partial_t f(t,x) & = \mathcal{B}_c f(t,x) + \mathcal{B}_b f(t,x) - \mathcal{D}f(t,x)\,, \qquad (t,x)\in (0,\infty)^2\,, \label{cecb1} \\
f(0,x) & = f^{in}(x) \geq 0 \,, \qquad x\in (0,\infty)\,, \label{cecb2}
\end{align}
\end{subequations}
where
\begin{subequations}\label{r}
\begin{equation}
\mathcal{B}_c f(x) := \frac{1}{2} \int_0^x E(x-y,y) K(x-y,y) f(x-y) f(y)\ dy\,,  \label{bc}
\end{equation}
\begin{equation}
\mathcal{B}_b f(x) := \frac{1}{2} \int_x^\infty \int_0^y b(x,y-z,z) (1-E(y-z,z)) K(y-z,z) f(y-z) f(z)\ dzdy\,, \label{bd}
\end{equation}
and
\begin{equation}
\mathcal{D} f(x) := \int_0^\infty K(x,y) f(x) f(y)\ dy\,. \label{d}
\end{equation}
\end{subequations}

The collision kernel $K(x,y)=K(y,x)\ge 0$ defines the rate at which particles of volumes $x$ and $y$ collide and the function $E(x,y)=E(y,x)\in [0,1]$ denotes the probability that the colliding particles of volumes $x$ and $y$ aggregate to form a larger one of volume $x+y$. Then, $1-E(x,y)$ is the probability that a collision event leads to the breakup of the colliding particles with a possible transfer of matter to form two or more particles. The daughter distribution function $b(x,y,z)\ge 0$ describes the average number of particles of volume $x$ produced during the breakage events resulting from the collision between particles of volumes $y$ and $z$. Specifically, the first integral $\mathcal{B}_c f(t,x)$ in \eqref{cecb1} denotes the formation of particles of volume $x$ due to coagulation events, and the third integral $\mathcal{D} f(t,x)$ represents the disappearance of particles of volume $x$ due to collisions. Moreover, the second integral $\mathcal{B}_b f(t,x)$ in \eqref{cecb1} accounts for the birth of particles of volume $x$ due to the collision-induced breakage of particles of volumes $y-z$ and $z$. We assume that there is no loss of matter during breakup and that no particle of volume exceeding the total volume of the colliding particles is created; that is, the daughter distribution function $b$ satisfies
\begin{subequations}\label{p40}
\begin{equation}
b(z,x,y)=b(z,y,x)\,, \qquad z\in (0,x+y)\,, \qquad b(z,x,y) = 0\,, \quad z>x+y\,, \label{p40a}
\end{equation}
and
\begin{equation}
\int_0^{x+y} z b(z,x,y)\ dz = x+y \label{p40b}
\end{equation}
\end{subequations}
for all $(x,y)\in (0,\infty)^2$. Since there is also no loss of matter during coagulation events, it is then expected that the total mass of the system, which is nothing but the first moment of $f$, is preserved by the dynamics of \eqref{cecb}, in the sense that
\begin{equation}
	\int_0^\infty x f(t,x)\ dx = \int_0^\infty x f^{in}(x)\ dx\,, \qquad t\ge 0\,. \label{mc}
\end{equation}
It is however well-known by now that, for the \emph{Smoluchowski coagulation equation} (SCE) ($E\equiv 1$), there is a loss of matter in finite time when the collision kernel grows superlinearly for large volumes, a phenomenon usually referred to as gelation, and \eqref{mc} ceases to be valid after a finite time, see \cite{BLL19, Dubo94b, EMP02, Leyv03} and the references therein. We shall return to this issue below when stating assumptions on the collision kernel.

On the one hand, it is worth mentioning that, for $E\equiv 1$, equation~\eqref{cecb1} reduces to the classical continuous SCE. In this case, the collision between a pair of particles of volumes $x$ and $y$ always leads to the coalescence of both into a single particle of volume $x+y$. Interestingly, if $b(x,y,z)=\delta_{y-x}+\delta_{z-x}$, then equation~\eqref{cecb1} again simplifies to the classical SCE with coagulation kernel $E K$. As already mentioned, since its derivation by Smoluchowski \cite{Smol16} and the pioneering contributions \cite{BaCa90, McLe62c, Melz57, Spou84, Stew89, Whit80}, it has been extensively studied in the mathematical literature and we refer to \cite{BLL19, Bert06, Dubo94b, Leyv03} for a more detailed account. On the other hand, when $E=0$, equation~\eqref{cecb} reduces to the \emph{nonlinear collision-induced fragmentation equation}  \cite{ChRe90, ChRe88, FTL88}. Most theoretical studies of this equation in the physical literature actually assume that there is no transfer of matter during collisions; that is, the fragment daughter distribution function $b$ is given by
\begin{equation}
b(z,x,y)= \mathbf{1}_{(0,x)}(z) B(z,x,y) + \mathbf{1}_{(0,y)}(z)B(z,y,x)\,, \label{z1ph}
\end{equation}
see \cite{ChRe90, ChRe88, ErPa07, KoKa00, KoKa06}. Here, $B(z,x,y)$ represents the breakup kernel (or breakage function) which describes the rate at which particles of volume $z$ are created by the collision between particles of volumes $x$ and $y$. Clearly, such a daughter distribution function satisfies \eqref{p40a} and complies with \eqref{p40b} if
\begin{equation*}
\int_0^x z B(z,x,y)\ dz = x\,, \qquad (x,y)\in (0,\infty)^2\,.
\end{equation*}

Besides the classical collision kernel
\begin{equation*}
	K_{\text{sm}}(x,y) = \left( x^{1/3} + y^{1/3} \right) \left( \frac{1}{x^{1/3}} + \frac{1}{y^{1/3}} \right)\,, \qquad (x,y)\in (0,\infty)^2\,,
\end{equation*}
derived by Smoluchowski in the seminal paper \cite{Smol16}, typical examples of collision kernels $K$ include the sum/product collision kernel
\begin{equation}
	K_{\zeta,\eta}(x,y) = x^\zeta y^\eta + x^\eta y^\zeta\,, \qquad (x,y)\in (0,\infty)^2\,, \label{ck}
\end{equation}
with $\zeta\le\eta\le 1$. As for the fragment distribution function $b$, model cases include 
\begin{equation}
	b_{\nu}(z,x,y) = (\nu+2) \frac{z^\nu}{(x+y)^{\nu+1}} \mathbf{1}_{(0,x+y)}(z)\,, \qquad (x,y,z)\in (0,\infty)^3\,, \label{bp1}
\end{equation}
see \cite{FTL88, VVF06}, or
\begin{equation}
	\tilde{b}_{\nu}(z,x,y) = (\nu+2) \frac{z^\nu}{x^{\nu+1}} \mathbf{1}_{(0,x)}(z) + (\nu+2) \frac{z^\nu}{y^{\nu+1}} \mathbf{1}_{(0,y)}(z)\,, \qquad (x,y,z)\in (0,\infty)^3\,, \label{bp2}
\end{equation}
see \cite{ErPa07}, for some $\nu>-2$. Observe that both $b_{\nu}$ and $\tilde{b}_{\nu}$ comply with the conservation of matter \eqref{p40} and that $\tilde{b}_\nu$ is of the form \eqref{z1ph}. For collision kernels $K_{\zeta,\eta}$, the well-posedness theory for \eqref{cecb} differs markedly  in the extreme cases $E\equiv 1$ (coagulation only) and $E\equiv 0$ (nonlinear fragmentation only). Indeed, when $E\equiv 1$, global weak solutions to \eqref{cecb} exist whatever the values of $\zeta$ and $\eta$, see \cite{BLL19, BaGiLaxx, CCGW15, CCWa15, Dubo94b, Stew89} and the references therein. In contrast, when $E\equiv 0$ and $b=\tilde{b}_{\nu}$, $\nu>-1$, global weak solutions to \eqref{cecb} exist only for $\zeta\ge 0$ and $\zeta+\eta\ge 1$, while weak solutions to \eqref{cecb} cannot be global for $\zeta\ge 0$ and $\zeta+\eta\in [0,1)$ and no non-zero weak solution to \eqref{cecb} exists for $\zeta<0$ \cite{ErPa07, GiLa}. 

It is thus of interest to study the well-posedness of \eqref{cecb} for intermediate values of $E$ and we focus in this paper on the global existence issue for collision kernels featuring a singularity ($\zeta<0$) or being locally bounded $(\zeta=0$) for small volumes. According to the above discussion, global existence to \eqref{cecb} is expected to require $E$ to be sufficiently close to one in a way which depends on the behaviour of the collision kernel $K$ for small volumes (monitored by the exponent $\zeta$ when $K=K_{\zeta,\eta}$). In other words, coagulation is assumed to be the governing mechanism, at least for small volumes, and we show the existence of global weak solutions to \eqref{cecb} for collision kernels $K$  featuring possibly a singularity for small volumes and growing, either linearly, or subquadratically, at infinity, see \eqref{p2} and \eqref{p3} below, respectively. We actually provide a positive lower bound on $E$ which guarantees the global existence of weak solutions to \eqref{cecb} and extends the results of \cite{BaGia, BaGib} to a broader class of collision kernels $K$ and daughter distribution functions $b$. We also identify a class of collision kernels $K$ (which includes $K_{0,1}$) for which \eqref{cecb} has global weak solutions under the sole assumption that $E$ ranges in $[0,1]$. 

Let us now describe the class of collision kernels $K$ and daughter distribution functions $b$ we are dealing with in this paper. We assume that the collision kernel $K$ is a measurable and symmetric function in $(0,\infty)^2$ and that there are $\alpha\in [0,1/2)$ and $k_1>0$ such that
\begin{equation}
0 \le K(x,y) \le \left\{
\begin{split}
& k_1 (xy)^{-\alpha}\,, \qquad (x,y)\in (0,1)^2\,, \\
& k_1 x^{-\alpha} y \,, \qquad (x,y)\in (0,1)\times (1,\infty)\,, \\
& k_1 x y^{-\alpha}\,, \qquad (x,y)\in (1,\infty) \times (0,1)\,, \\
& k_1 x y\,, \qquad (x,y)\in (1,\infty)^2\,.
\end{split} 
\right. \label{p1}
\end{equation}
The collision kernel $K_{\zeta,\eta}$ defined in \eqref{ck} obviously satisfies \eqref{p1} with $\alpha=\max\{-\zeta,0\}$ when $-1/2<\zeta\le\eta\le 1$. As for the behaviour at infinity, we shall distinguish two different cases due to the already mentioned possible occurrence of the gelation phenomenon when the collision kernel increases superlinearly for large volumes. Specifically, we shall additionally assume  that, either there is $k_2>0$ such that
\begin{equation}
K(x,y) \le k_2 (x+y)\,, \qquad (x,y)\in (1,\infty)^2\,, \label{p2}
\end{equation}
or	
\begin{subequations}\label{p3}
\begin{equation}
K(x,y) \le \left\{
\begin{split}
& x^{-\alpha} r(y)\,, \qquad (x,y)\in (0,1)\times (1,\infty)\,, \\
& r(x) y^{-\alpha}\,, \qquad (x,y)\in (1,\infty)\times (0,1)\,, \\
& r(x) r(y)\,, \qquad (x,y)\in (1,\infty)^2\,, 
\end{split}
\right. \label{p3a}
\end{equation}
for some function $r: (0,\infty)\to [1,\infty)$ satisfying
\begin{equation}
\sup_{x>0} \left\{ \frac{r(x)}{1+x} \right\}< \infty \;\;\text{ and }\;\; \lim_{x\to\infty} \frac{r(x)}{x} = 0\,. \label{p3b}
\end{equation}
\end{subequations}
Clearly $K_{\text{sm}}$ satisfies \eqref{p2}, as well as $K_{\zeta,\eta}$ when $\zeta+\eta\le 1$, while $K_{\zeta,\eta}$ satisfies \eqref{p3} for $-1/2<\zeta\le\eta<1$. Roughly speaking, for a suitable class of initial conditions $f^{in}$, probabilities $E$, and daughter distribution functions $b$ (depending on $\alpha$), we establish in this paper the existence of mass-conserving solutions when $K$ satisfies \eqref{p1} and \eqref{p2} and the existence of weak solutions when $K$ satisfies \eqref{p1} and \eqref{p3}. Moreover, we show the uniqueness of mass-conserving solutions under additional restrictions on $f^{in}$ and $b$. In particular, we are able herein to handle collision kernels featuring higher singularities for small volumes than in \cite{BaGib}, thereby extending the existence results obtained in \cite{BaGib} for collision kernels $K$ being bounded from above by a multiple of the kernel $\tilde{K}_{\sigma,\eta}$ defined by
\begin{equation}
\tilde{K}_{\sigma,\eta}(x,y) := \frac{(1+x)^{\eta}(1+y)^{\eta}}{(x+y)^{\sigma}}\,, \qquad (x,y)\in (0,\infty)^2\,, \label{ckbg}
\end{equation}
with $\sigma \in [0, 1)$ and $\eta \in [0,(2+\sigma)/2)$. Indeed, $\tilde{K}_{\sigma,\eta}$ satisfies \eqref{p1} with $\alpha=\sigma/2$ and \eqref{p3} with $r(x) = x^{(2\eta-\sigma)/2}$, $x>0$. We also extend the existence result established in \cite{BaGia} for collision kernels $K_{\zeta,\eta}$ given by \eqref{ck} when $0 < \zeta \leq \eta < 1$, including in particular the constant collision kernel $K_{0,0}$ which is excluded from the analysis performed in \cite{BaGia, BaGib}. The uniqueness result we obtain herein also encompasses the one proved in \cite{BaGib} which is restricted to collision kernels satisfying \eqref{ckbg} for $\eta=0$ and $\sigma\in (0,1/2)$. Let us finally mention that, when $K$ is given by \eqref{ck} with $\eta=1-\zeta\in (0,1)$ and $b(z,x,y) = 2/(x+y)$, the existence and uniqueness of a global mass-conserving classical solution emanating from an initial condition $f^{in}\in C([0,\infty))\cap L^1((0,\infty),(1+x^{\max\{1+\mu,2-\mu\}})dx)$ are shown in \cite{BaGi20} and we supplement it here with the existence and uniqueness of a global mass-conserving weak solution with initial condition $f^{in}\in L^1((0,\infty),(1+x^2)dx)$.

\bigskip

Let us now outline the contents of the paper. We devote the remainder of the introduction to the statements of the assumptions on $E$ and $b$ which we need to establish the existence of global weak solutions to \eqref{cecb}. We already emphasize here that, as in \cite{BaGia, BaGib}, the main assumption is that coagulation is the governing mechanism, in the sense that $E$ has to be sufficiently close to one in a way which is controlled by the behaviour of the daughter distribution function $b$ for small volumes, see \eqref{p7b} below. We next gather the main results of this paper in Section~\ref{secM}, which deal with the existence and uniqueness of weak solutions to \eqref{cecb}. We also summarize there the outcome of our results for the collision kernels defined in \eqref{ck} and \eqref{ckbg} and the daughter distribution functions defined in \eqref{bp1} and \eqref{bp2}. The existence results are proved in Section~\ref{sec2} and the proofs are based on the weak $L^1$-compactness method which was originally developed in \cite{Stew89} for the CFE. To handle the possible singularity of the collision kernel $K$, we adapt the techniques developed in \cite{BLL19, BaGiLaxx, CCGW15, CCWa15, EMRR05} for the CFE with singular coagulation kernels, while the assumption on $E$ (dominating coagulation) allows us to control the behaviour of the distribution function for small volumes. The uniqueness proof is supplied in Section~\ref{sec3}.

\bigskip
 
 Coming back to equation~\eqref{cecb}, let $K$ be a collision kernel satisfying \eqref{p1} for some $\alpha\in [0,1/2)$. Concerning the daughter distribution function $b$, we assume that it is a measurable function satisfying \eqref{p40} but the analysis requires several additional assumptions which depend on the value of $\alpha$. We begin with a uniform integrability property and assume that there are a non-decreasing function $\omega: (0,\infty)\to (0,\infty)$ and $\theta\in (0,1)$ such that 
\begin{subequations}\label{p5}
\begin{equation}
\lim_{\xi\to 0} \omega(\xi) = 0 \label{p5b}
\end{equation}
and
\begin{equation}
\int_0^{x+y} z^{-\alpha} \mathbf{1}_A(z) b(z,x,y)\ dz \le \omega(|A|) (x+y)^{-\alpha} \left( x^{-\theta} + y^{-\theta} \right)\,, \qquad (x,y)\in (0,\infty)^2\,, \label{p5a}
\end{equation}
for all measurable subsets $A$ in $(0,\infty)$ with finite (Lebesgue) measure.
\end{subequations}
We finally require a control on the growth of $b$ for small sizes and separate the cases $\alpha=0$ and $\alpha>0$. 
\begin{itemize}
\item[(A)] If $\alpha=0$, then there are  $\beta_0\ge 1$ and $\beta_{-\theta}\ge 2^{-\theta}$ such that 
\begin{subequations}\label{p4}
\begin{align}
\int_0^{x+y} b(z,x,y)\ dz & \le \beta_0\,, \qquad (x,y)\in (0,\infty)^2\,, \label{p4a} \\
\int_0^{x+y} z^{-\theta} b(z,x,y)\ dz & \le \frac{\beta_{-\theta}}{2} (x^{-\theta}+y^{-\theta})\,, \qquad (x,y)\in (0,\infty)^2\,. \label{p4b}
\end{align}
\end{subequations}
\item[(B)] If $\alpha\in (0,1/2)$, then 
\begin{subequations}\label{p6}
\begin{equation}
\theta\in (0,\alpha] \,, \label{p6b}
\end{equation}	
and there is  $\beta_{-2\alpha}\ge 1$ such that
\begin{equation}
\int_0^{x+y} z^{-2\alpha} b(z,x,y)\ dz \le \beta_{-2\alpha} (x+y)^{-2\alpha}\,, \qquad (x,y)\in (0,\infty)^2\,. \label{p6a}
\end{equation}
\end{subequations}
Since $\alpha>0$, it readily follows from \eqref{p6a} that it also holds
\begin{align}
\int_0^{x+y} z^{-\alpha} b(z,x,y)\ dz & \le \beta_{-2\alpha} (x+y)^{-\alpha}\,, \qquad (x,y)\in (0,\infty)^2\,, \label{p6c} \\
\int_0^{x+y}  b(z,x,y)\ dz & \le \beta_{-2\alpha}\,, \qquad (x,y)\in (0,\infty)^2\,. \label{p6d}
\end{align}
\end{itemize}
Let us first mention that the daughter distribution function $b_\nu$ introduced in \eqref{bp1} satisfies \eqref{p5}, \eqref{p4}, and \eqref{p6} for appropriate values of  $\nu$ (depending on $\alpha$). Indeed, if $\alpha\in (0,1/2)$, then $b_\nu$ satisfies \eqref{p5} and \eqref{p6} for $\nu>2\alpha-1$ with 
\begin{equation*}
\omega(\xi) = (\nu+2) \left( \frac{p-1}{p(\nu+1-\alpha)-1} \right)^{(p-1)/p} \xi^{1/p} \,, \quad \theta = \frac{1}{p}\,, \quad \beta_{-2\alpha} = \frac{\nu+2}{\nu+1-2\alpha}\,,
\end{equation*}
where the parameter $p>\max\{1/\alpha,1/(\nu+1-\alpha)\}$ can be chosen arbitrarily provided that it is large enough. Similarly, if $\alpha=0$, then $b_\nu$ satisfies \eqref{p5} and \eqref{p4} for $\nu>-1$ with
\begin{equation*}
\omega(\xi) = (\nu+2) \left( \frac{p-1}{p(\nu+1)-1} \right)^{(p-1)/p} \xi^{1/p} \,, \quad \theta = \frac{1}{p}\,, \quad \beta_0 = \frac{\nu+2}{\nu+1}\,, \quad \beta_{-\theta} = \frac{\nu+2}{\nu+1-\theta}\,,
\end{equation*}
provided the parameter $p>1/(\nu+1)$ is large enough.

Since the daughter distribution function $\tilde{b}_\nu$ introduced in \eqref{bp2} features a higher singularity for small volumes, it only fits in the analysis performed in this paper for $\alpha=0$. Indeed, if $\alpha=0$,  then $\tilde{b}_\nu$ satisfies  \eqref{p5} and \eqref{p4} for $\nu>-1$ with
\begin{equation*}
\omega(\xi) = (\nu+2) \left( \frac{p-1}{p(\nu+1)-1} \right)^{(p-1)/p} \xi^{1/p} \,, \quad \theta = \frac{1}{p}\,, \quad \beta_0 = 2\frac{\nu+2}{\nu+1}\,, \quad \beta_{-\theta} = 2\frac{\nu+2}{\nu+1-\theta}\,,
\end{equation*}
provided the parameter $p>1/(\nu+1)$ is large enough.

Finally, the probability $E$ that a collision event leads to coalescence is a measurable function in $(0,\infty)^2$ such that 
\begin{subequations}\label{p7}
\begin{equation}
0 \le E(x,y) = E(y,x) \le 1 \,, \qquad (x,y)\in (0,\infty)^2\,, \label{p7a}
\end{equation}
and we assume that coagulation is the dominant mechanism for small sizes in the following sense:
\begin{equation}
E(x,y) \ge \max\left\{ 0 , \frac{\beta_{-2\alpha}-2^{1+2\alpha}}{\beta_{-2\alpha}-1} \right\}\,, \qquad (x,y)\in (0,1)^2\,, \label{p7b}
\end{equation}
where $\beta_{-2\alpha}=\beta_0$ is defined in  \eqref{p4a} for $\alpha=0$ and in \eqref{p6} for $\alpha\in (0,1/2)$.
\end{subequations}

\begin{remark}
Since we handle a rather general class of collision kernels in this paper, see \eqref{p1}, the lower bound \eqref{p7b} derived on $E$ might not be sharp. A more precise study involving specific choices of collision kernels (such as $K_{\zeta,\eta}$ in \eqref{ck}) and daughter distribution functions (such as $b_\nu$ in \eqref{bp1} or $\tilde{b}_\nu$ in \eqref{bp2}) is likely to be needed to identify sharp threshold values of $E$, in particular when $E$ is assumed to be constant. Also, the assumption \eqref{p1} does not take into account the vanishing properties of $K$ for small volumes (corresponding to $\zeta>0$ in the example $K_{\zeta,\eta}$) which are helpful to obtain global existence when $E\equiv 0$. We hope to investigate further these issues in the near future.
\end{remark}

\bigskip

\noindent\textbf{Notation.} Given a non-negative measurable function $V$ on $(0,\infty)$, we set $X_V := L^1((0,\infty),V(x) dx)$ and 
\begin{equation*}
M_V(h) := \int_0^\infty h(x) V(x)\ dx\,, \qquad h\in X_V\,.
\end{equation*}
We also denote the positive cone of $X_V$ by $X_V^+$, while $X_{V,w}$ stands for the space $X_V$ endowed with its weak topology. When $V(x)=V_m(x) := x^m$, $x\in (0,\infty)$, for some $m\in\mathbb{R}$, we set $X_m := X_{V_m}$ and
\begin{equation*}
M_m(h) := M_{V_m}(h) = \int_0^\infty x^m h(x)\ dx\,, \qquad h\in X_m\,.
\end{equation*}

\section{Main results}\label{secM}

 We begin with collision kernels featuring a singularity for small volumes ($\alpha\in (0,1/2)$) and report the following existence result, which can be seen as an extension of \cite{BaGib} to a broader class of collision kernels, including Smoluchowski's collision kernel. In fact, the singularity of $K$ allowed in \cite{BaGib} is of the form $(x+y)^{-\alpha}$ as $(x,y)\to (0,0)$, while stronger singularities of the form $(xy)^{-\alpha}$ or $x^{-\alpha}+y^{-\alpha}$ as $(x,y)\to (0,0)$ are included in our analysis.

\begin{theorem}[Existence: $\alpha\in (0,1/2)$] \label{tha1}
Let $\alpha\in (0,1/2)$ and consider an initial condition $f^{in} \in X_{-2\alpha}\cap X_1^+$. We assume that $K$, $b$, and $E$ satisfy \eqref{p1}, \eqref{p40}, \eqref{p5}, \eqref{p6}, and \eqref{p7}.
\begin{itemize}
\item[(a)] If $K$ satisfies also \eqref{p3}, then there is at least one global weak solution $f$ to \eqref{cecb}; that is, 
\begin{equation}
f\in C([0,\infty),X_{-\alpha,w})\cap  L^\infty((0,T),X_{-2\alpha}\cap X_1^+) \;\;\text{ for all }\;\; T>0\,, \label{pa1}
\end{equation}
is such that
\begin{equation}
M_1(f(t)) \le M_1(f^{in})\,, \qquad t\ge 0\,, \label{pa2}
\end{equation}
and $f$ satisfies 
\begin{equation}
\int_0^\infty \phi(x) [f(t,x)-f^{in}(x)]\ dx = \frac{1}{2} \int_0^t \int_0^\infty \int_0^\infty \zeta_\phi(x,y) K(x,y) f(s,x) f(s,y)\ dydxds \label{pa3}
\end{equation}
for all $t>0$ and $\phi\in L^\infty(0,\infty)$, where
\begin{equation*}
\zeta_\phi(x,y) := E(x,y) \phi(x+y) + (1-E(x,y)) \int_0^{x+y} \phi(z) b(z,x,y)\ dz - \phi(x) - \phi(y)
\end{equation*}
for $(x,y)\in (0,\infty)^2$.

\item[(b)] If $K$ satisfies also \eqref{p2}, then there is at least one global mass-conserving weak solution $f$ to \eqref{cecb}; that is, 
\begin{equation}
f\in C([0,\infty),X_{-\alpha,w}\cap X_{1,w})\cap L^\infty((0,T),X_{-2\alpha}\cap X_1^+) \;\;\text{ for all }\;\; T>0\,,\label{pa4}
\end{equation}
\begin{equation}
M_1(f(t)) = M_1(f^{in})\,, \qquad t\ge 0\,, \label{pa5}
\end{equation}
and f satisfies \eqref{pa3}.

\item[(c)] If $K$ satisfies also \eqref{p2} and $f^{in}\in X_2$, then there is at least one global mass-conserving weak solution $f$ to \eqref{cecb} satisfying \eqref{pa3}, \eqref{pa4}, \eqref{pa5} and such that $f\in L^\infty((0,T),X_2)$ for each $T>0$.
\end{itemize} 
\end{theorem}

The next result extends \cite{BaGia, BaGi20} and, in contrast to Theorem~\ref{tha1}, applies to daughter distribution functions $b$ considered in \cite{ChRe90, ErPa07}, see \eqref{bp2}.

\begin{theorem}[Existence: $\alpha=0$]\label{tha2}
Let $\alpha=0$ and consider an initial condition $f^{in} \in X_{-\theta}\cap X_1^+$. We assume that $K$, $b$, and $E$ satisfy \eqref{p1}, \eqref{p40}, \eqref{p5}, \eqref{p4}, and \eqref{p7}.
\begin{itemize}
\item[(a)] If $K$ satisfies also \eqref{p3}, then there is at least one global weak solution $f$ to \eqref{cecb}; that is, $f$ satisfies \eqref{pa1}, \eqref{pa2}, and \eqref{pa3}. 
\item[(b)] If $K$ satisfies also \eqref{p2}, then there is at least one global mass-conserving weak solution $f$ to \eqref{cecb}; that is, $f$ satisfies \eqref{pa3}, \eqref{pa4}, and \eqref{pa5}.

\item[(c)] If $K$ satisfies also \eqref{p2} and $f^{in}\in X_2$, then there is at least one global mass-conserving weak solution $f$ to \eqref{cecb} satisfying \eqref{pa3}, \eqref{pa4}, \eqref{pa5} and such that $f\in L^\infty((0,T),X_2)$ for each $T>0$.
\end{itemize}
\end{theorem}

 We next provide a variant of Theorem~\ref{tha2} where we relax the integrability properties of $f^{in}$ for small sizes at the expense of slightly stronger assumptions on the daughter distribution function $b$. 
	
\begin{theorem}[Existence: $\alpha=0$]\label{tha2b}
Let $\alpha=0$ and consider an initial condition $f^{in} \in X_0 \cap X_1^+$. We assume that $K$, $b$, and $E$ satisfy \eqref{p1}, \eqref{p40}, \eqref{p5}, \eqref{p4a}, \eqref{p7}, and: either there is $\beta_{-\theta}'\ge 2$ such that
\begin{equation}
\int_0^{x+y} z^{-\theta} b(z,x,y)\ dz \le \frac{\beta_{-\theta}'}{2} (x+y)^{-\theta} \,, \qquad (x,y)\in (0,\infty)^2\,, \label{paa4a}
\end{equation}
\begin{subequations}\label{paa4}
or there are $\beta_{-\theta}'\ge 2$ and a non-negative measurable function $\bar{b}$ such that
\begin{align}
b(z,x,y) = \bar{b}(z,x,y) \mathbf{1}_{(0,x)}(z) + \bar{b}(z,y,x) \mathbf{1}_{(0,y)}(z) \,, \qquad (x,y,z)\in (0,\infty)^3\,, \label{paa4b} \\
\int_0^x z \bar{b}(z,x,y)\ dz = x\,, \qquad (x,y)\in (0,\infty)^2\,, \label{paa4c} \\
\int_0^x z^{-\theta} \bar{b}(z,x,y)\ dz \le \frac{\beta_{-\theta}'}{2} x^{-\theta}\,, \qquad (x,y)\in (0,\infty)^2\,. \label{paa4d}
\end{align}
\end{subequations}

\begin{itemize}
	\item[(a)] If $K$ satisfies also \eqref{p3}, then there is at least one global weak solution $f$ to \eqref{cecb}; that is, $f$ satisfies \eqref{pa1}, \eqref{pa2}, and \eqref{pa3}. 
	\item[(b)] If $K$ satisfies also \eqref{p2}, then there is at least one global mass-conserving weak solution $f$ to \eqref{cecb}; that is, $f$ satisfies \eqref{pa3}, \eqref{pa4}, and \eqref{pa5}.		
	\item[(c)] If $K$ satisfies also \eqref{p2} and $f^{in}\in X_2$, then there is at least one global mass-conserving weak solution $f$ to \eqref{cecb} satisfying \eqref{pa3}, \eqref{pa4}, \eqref{pa5} and such that $f\in L^\infty((0,T),X_2)$ for each $T>0$.
	\end{itemize}
\end{theorem}

Clearly, \eqref{paa4a} and \eqref{paa4d} separately imply \eqref{p4b} with $\beta_{-\theta}=\beta_{-\theta}'$, so that these two assumptions are indeed stronger than \eqref{p4b}. But, no negative finite moment is required for the initial condition in Theorem~\ref{tha2b}, in contrast to Theorem~\ref{tha2}. Relaxing this assumption is possible thanks to a refined version of the de la Vall\'ee Poussin theorem \cite{dlVP15} which we establish in Lemma~\ref{leap1}. It is also worth pointing out that Theorem~\ref{tha2b} also applies to the daughter distribution functions $b_\nu$ and $\tilde{b}_\nu$ defined in \eqref{bp1} and \eqref{bp2}, respectively, for $\nu>-1$. Indeed, $b_\nu$ and $\tilde{b}_\nu$ satisfy \eqref{paa4a} and \eqref{paa4}, respectively, with $\beta_{-\theta}'=2(\nu+2)/(\nu+1-\theta)$, recalling that $\theta=1/p<\nu+1$. We finally emphasize that Theorem~\ref{tha2b} applies to the constant collision kernel which is excluded from the analysis in \cite{BaGi20, BaGia}.

\begin{remark}\label{rema3}
As already noticed in \cite{BaGi20}, the daughter distribution function
\begin{equation}
b(z,x,y) = \frac{2}{x+y} \mathbf{1}_{(0,x+y)}(z)\,, \qquad (x,y,z)\in (0,\infty)^3\,, \label{p300}
\end{equation}	
has peculiar properties. Indeed, in that case, $\beta_0=2$ and \eqref{p7b} is always satisfied due to \eqref{p7a}. In particular, Theorem~\ref{tha2} provides an existence result for the collision-induced breakage equation which corresponds to the choice $E\equiv 0$ in \eqref{cecb}. 
\end{remark}

 In the same vein, an existence result for arbitrary non-negative $E$ is available for a specific class of collision kernels and is reported next. 

\begin{theorem}[$K(x,y)\le k_0 (x+y)$]\label{tha4}
Assume that the collision kernel $K$ is a measurable and symmetric function in $(0,\infty)^2$ and that there is $k_0>0$ such that 
\begin{equation}
0 \le K(x,y) \le k_0 (x+y)\,, \qquad (x,y)\in (0,\infty)^2\,. \label{p400}
\end{equation}
Assume also that $b$ satisfies \eqref{p40}, \eqref{p5}, and \eqref{p4}, while $E$ satisfies \eqref{p7a}. Given an initial condition $f^{in} \in X_{-\theta}\cap X_1^+$, there is at least one global mass-conserving weak solution $f$ to \eqref{cecb} satisfying \eqref{pa3}, \eqref{pa4}, \eqref{pa5}.
\end{theorem}

It is worth pointing out that, unlike Theorems~\ref{tha1} and ~\ref{tha2}, Theorem~\ref{tha4} does not require a positive lower bound on $E$ but only that it is non-negative. As in the case described in Remark~\ref{rema3}, it provides an existence result for the collision-induced breakage equation which corresponds to the choice $E\equiv 0$ in \eqref{cecb}. We refer to the companion paper \cite{GiLa} for a more complete study of this model.

We finally supplement the above existence results with the uniqueness of mass-conserving weak solutions to \eqref{cecb} having a finite second moment.

\begin{theorem}[Uniqueness]\label{tha5}
Let $\alpha\in [0,1/2)$ and assume that $K$, $b$, and $E$ satisfy \eqref{p1}, \eqref{p2}, \eqref{p40}, and \eqref{p7a}, respectively. Assume further that there is $B_{-\alpha}>1$ such that
\begin{equation}
\int_0^{\min\{ 1 , x+y\}} z^{-\alpha} b(z,x,y)\ dz \le B_{-\alpha} \left( \min\{1, x+y\} \right)^{-\alpha}\,, \qquad (x,y)\in (0,\infty)^2\,. \label{p500}
\end{equation}
Then, given an initial condition $f^{in}\in X_{-2\alpha}\cap X_1^+\cap X_2$, there is at most one weak solution $f$ to \eqref{cecb} satisfying \eqref{pa3}, \eqref{pa4}, \eqref{pa5} and such that $f\in L^\infty((0,T),X_{-2\alpha}\cap X_2)$ for each $T>0$.
\end{theorem}

The proof of Theorem~\ref{tha5} relies on the control of the difference of two solutions to \eqref{cecb} in a suitable weighted $L^1$-space and follows the lines of the uniqueness proofs performed in \cite{EMRR05, Stew90b}. Let us also mention that $b_\nu$ defined in \eqref{bp1} satisfies \eqref{p500} for $\alpha\in [0,1/2)$ and $\nu>2\alpha-1$ with $B_{-\alpha} =(\nu+2)/(\nu+1-\alpha)$, while $\tilde{b}_\nu$ satisfies \eqref{p500} for $\alpha=0$ and $\nu>-1$ with $B_0 = 2(\nu+2)/(\nu+1)$.

\medskip

To illustrate the outcome of the above results which have been derived under rather general assumptions on $K$ and $b$, we end up this section with the existence and uniqueness of global mass-conserving weak solutions to \eqref{cecb} when the collision kernel is given by \eqref{ck} or \eqref{ckbg} and the daughter distribution function $b$ by \eqref{bp1} or \eqref{bp2}.  Let us begin with the sum/product collision kernel $K=K_{\zeta,\eta}$ defined in \eqref{ck}.

\begin{corollary}\label{cora7}
Consider $-1/2<\zeta \le \eta \le 1$ such that $\zeta+\eta\le 1$ and $\nu>2\alpha-1$, where $\alpha := \max\{0,-\zeta\}$. Let $f^{in}\in X_{-2\alpha}\cap X_1^+ \cap X_2$. There is a unique global mass-conserving weak solution to \eqref{cecb} in the following cases:
\begin{itemize}
	\item [(i)] $\zeta\ge 0$, $K=K_{\zeta,\eta}$, and, either $b=b_\nu$ and
	\begin{equation}
		E(x,y) \ge \max\{0,-\nu\}\,, \qquad (x,y)\in (0,1)^2\,, \label{e1}
	\end{equation}
or $b=\tilde{b}_\nu$ and 
\begin{equation}
	E(x,y) \ge \frac{2}{\nu+3}\,, \qquad (x,y)\in (0,1)^2\,; \label{e2}
\end{equation}
\item [(ii)] $\zeta<0$, $K=K_{\zeta,\eta}$, $b=b_\nu$, and
\begin{equation}
	E(x,y) \ge \max\left\{ 0 , \frac{\nu+2 - (\nu+1+2\zeta) 2^{1-2\zeta}}{1-2\zeta} \right\}\,, \qquad (x,y)\in (0,1)^2\,. \label{e3}
\end{equation}
\end{itemize}
\end{corollary}

We point out that the right-hand side of \eqref{e3} is always positive in the physically relevant case $\nu\in (-1,0]$. The existence statement in Corollary~\ref{cora7}~(i) and~(ii) readily follows from Theorem~\ref{tha2b}~(b)-(c) and Theorem~\ref{tha1}~(b)-(c), respectively, while the uniqueness assertion is a consequence of Theorem~\ref{tha5}. 

A similar result is obtained in the same way for the collision kernel $K=\tilde{K}_{\sigma,\eta}$ defined in \eqref{ckbg} and considered in \cite{BaGib}.

\begin{corollary}\label{cora8}
	Consider $\sigma\in [0,1)$, $\eta\in [0,(\sigma+1)/2]$, and $\nu>\sigma-1$. Let $f^{in}\in X_{-\sigma}\cap X_1^+ \cap X_2$. There is a unique global mass-conserving weak solution to \eqref{cecb} in the following cases:
	\begin{itemize}
		\item [(i)] $\sigma= 0$, $K=\tilde{K}_{0,\eta}$, and, either $b=b_\nu$ and $E$ satisfies \eqref{e1}, or $b=\tilde{b}_\nu$ and $E$ satisfies \eqref{e2};
		\item [(ii)] $\sigma<0$, $K=\tilde{K}_{\sigma,\eta}$, $b=b_\nu$, and
		\begin{equation}
			E(x,y) \ge \max\left\{ 0 , \frac{\nu+2 - (\nu+1-\sigma) 2^{1+\sigma}}{1+\sigma} \right\}\,, \qquad (x,y)\in (0,1)^2\,. \label{e4}
		\end{equation}
	\end{itemize}
\end{corollary}

Let us now turn to global weak solutions and apply Theorem~\ref{tha1}~(a) and Theorem~\ref{tha2b}~(a), first to the sum/product collision kernel $K=K_{\zeta,\eta}$ defined in \eqref{ck}.

\begin{corollary}\label{cora9}
	Consider $-1/2<\zeta \le \eta < 1$ and $\nu>2\alpha-1$, where $\alpha := \max\{0,-\zeta\}$. Let $f^{in}\in X_{-2\alpha}\cap X_1^+$. There is a global weak solution to \eqref{cecb} in the following cases:
	\begin{itemize}
		\item [(i)] $\zeta\ge 0$, $K=K_{\zeta,\eta}$, and, either $b=b_\nu$ and $E$ satisfies \eqref{e1}, or $b=\tilde{b}_\nu$ and $E$ satisfies \eqref{e2};
		\item [(ii)] $\zeta<0$, $K=K_{\zeta,\eta}$, $b=b_\nu$, and $E$ satisfies \eqref{e3}.
	\end{itemize}
\end{corollary}

When $b=b_\nu$, Corollary~\ref{cora9} extends the existence result in \cite{BaGia} to a broader class of collision kernels, including the case $\zeta\in (-1/2,0]$ which was not handled there. The case $b=\tilde{b}_\nu$ is not considered in \cite{BaGia} and thus seems to be new.

We end up with the global existence of weak solutions to \eqref{cecb} when the collision kernel $K=\tilde{K}_{\sigma,\eta}$ is defined in \eqref{ckbg}.

\begin{corollary}\label{cora10}
	Consider $\sigma\in [0,1)$, $\eta\in [0,(\sigma+2)/2)$, and $\nu>\sigma-1$. Let $f^{in}\in X_{-\sigma}\cap X_1^+$. There is a global weak solution to \eqref{cecb} in the following cases:
	\begin{itemize}
		\item [(i)] $\sigma= 0$, $K=\tilde{K}_{0,\eta}$, and, either $b=b_\nu$ and $E$ satisfies \eqref{e1}, or $b=\tilde{b}_\nu$ and $E$ satisfies \eqref{e2};
		\item [(ii)] $\sigma<0$, $K=\tilde{K}_{\sigma,\eta}$, $b=b_\nu$, and $E$ satisfies \eqref{e4}.
	\end{itemize}
\end{corollary}

Corollary~\ref{cora10}~(ii) is already obtained in \cite{BaGib} under the additional constraint  $\eta\in [\sigma,\sigma+1)\cap [0,1)$ which we relax here. Corollary~\ref{cora10}~(i) seems to be new.

\section{Existence}\label{sec2}

We consider a collision kernel $K$ satisfying \eqref{p1} for some $\alpha\in [0,1/2)$, a daughter distribution function $b$ satisfying \eqref{p40}, and a coagulation probability $E$ satisfying \eqref{p7a}. We also fix a non-negative initial condition 
\begin{equation}
f^{in}\in X_{-2\alpha}\cap X_1^+\,, \label{pi1}
\end{equation}
and set
\begin{equation}
\varrho := M_1(f^{in})\,. \label{pi2}
\end{equation}

For $n\ge 1$, we define
\begin{equation}
K_n(x,y) := \min\{ n , K(x,y)\} \mathbf{1}_{(0,n)}(x+y)\,, \qquad (x,y)\in (0,\infty)^2\,, \label{p20a}
\end{equation}
and
\begin{equation}
f_n^{in} = f^{in} \mathbf{1}_{(0,n)}\,. \label{p20b}
\end{equation}

The starting point of the proof of Theorems~\ref{tha1} and~\ref{tha2} is the well-posedness of the coagulation equation with collisional breakage \eqref{cecb} with $K_n$ instead of $K$. More precisely, we look for a solution $f_n$ to
\begin{subequations}\label{cecbt}
	\begin{align}
	\partial_t f_n(t,x) & = \mathcal{B}_{c,n} f_n(t,x) + \mathcal{B}_{b,n} f_n(t,x) - \mathcal{D}_n f_n(t,x)\,, \qquad (t,x)\in (0,\infty)\times (0,n)\,, \label{cecbt1} \\
	f_n(0,x) & = f_n^{in}(x) \,, \qquad x\in (0,n)\,, \label{cecbt2}
	\end{align}
\end{subequations}
where
\begin{subequations}\label{rt}
\begin{equation}
\mathcal{B}_{c,n} f(x) := \frac{1}{2} \int_0^x E(x-y,y) K_n(x-y,y) f(x-y) f(y)\ dy\,,  \label{bct}
\end{equation}
\begin{equation}
\mathcal{B}_{b,n} f(x) := \frac{1}{2} \int_x^n \int_0^y b(x,y-z,z) (1-E(y-z,z)) K_n(y-z,z) f(y-z) f(z)\ dzdy\,, \label{bdt}
\end{equation}
and
\begin{equation}
\mathcal{D}_n f(x) := \int_0^{n-x} K_n(x,y) f(x) f(y)\ dy \label{dt}
\end{equation}
for $x\in (0,n)$. 
\end{subequations}

\begin{proposition}\label{propb1} \cite{Walk02}
	Let $n\ge 1$. There is a unique non-negative solution $f_n\in C^1([0,T_n),L^1(0,n))$ to \eqref{cecbt} defined up to some maximal existence time $T_n\in (0,\infty]$. In addition, if $T_n<\infty$, then
	\begin{equation}
	\lim_{t\to T_n} \int_0^n f_n(t,x)\ dx = \infty\,. \label{p200}
	\end{equation}
	Furthermore, 
	\begin{equation}
	\int_0^n x f_n(t,x)\ dx = \int_0^n x f_n^{in}(x)\ dx \le \varrho\,, \qquad t\in [0,T_n)\,. \label{p21}
	\end{equation}
\end{proposition}

\begin{proof}
	Since $K_n$ is bounded and supported in $(0,n)^2$ by \eqref{p1} and \eqref{p20a}, the existence and uniqueness of a non-negative strong solution $f_n\in C^1([0,T_n),L^1(0,n))$ to \eqref{cecbt} and the blowup property \eqref{p200} readily follow from \cite[Theorems~2.2 and~2.4]{Walk02}. As for the conservation of mass \eqref{p21}, it is a consequence of \eqref{pi1}, \eqref{p20b}, and \cite[Theorem~2.8]{Walk02}.
\end{proof}

We extend $f_n$ to $(0,T_n)\times (0,\infty)$ by setting $f_n(t,x)=0$ for $(t,x)\in (0,T_n)\times (n,\infty)$, still denoting this extension by $f_n$. It easily follows from \eqref{p20a}, \eqref{cecbt}, \eqref{rt}, and Proposition~\ref{propb1} that
\begin{equation}
\frac{d}{dt} \int_0^\infty \phi(x) f_n(t,x)\ dx = \frac{1}{2} \int_0^\infty \int_0^\infty \zeta_\phi(x,y) K_n(x,y) f_n(t,x) f_n(t,y)\ dydx \label{p9}
\end{equation}
for all $t\in (0,T_n)$ and $\phi\in L^\infty(0,\infty)$.

\bigskip

Throughout this section, $C$ and $(C_i)_{i\ge 1}$ are positive constants depending only on $K$, $b$, $E$, and $f^{in}$. Dependence upon additional parameters is indicated explicitly.

\subsection{Existence: $\alpha\in (0,1/2)$}\label{sec2.1}

\newcounter{NCEx}

In this section, besides \eqref{p1}, \eqref{p40}, and  \eqref{p7}, we further assume that $\alpha\in (0,1/2)$, while the daughter distribution function $b$ satisfies \eqref{p5} and \eqref{p6}. 

\begin{lemma}\label{lemb3}
\refstepcounter{NCEx}\label{cst1} For all $n\ge 1$, $T_n=\infty$ and, for any $T>0$, there is $C_{\ref{cst1}}(T)>0$ such that 
\begin{align}
M_{-2\alpha}(f_n(t)) & \le C_{\ref{cst1}}(T)\,, \qquad t\in [0,T]\,, \ n\ge 1 \,, \label{p25a} \\
M_{-\alpha}(f_n(t)) & \le C_{\ref{cst1}}(T)\,, \qquad t\in [0,T]\,, \ n\ge 1\,. \label{p25b}
\end{align}
\end{lemma}

\begin{proof}
Owing to the convexity of $x\mapsto x^{-2\alpha}$ and \eqref{p6},
\begin{align*}
\zeta_{-2\alpha}(x,y) & = E(x,y) (x+y)^{-2\alpha} + (1-E(x,y)) \int_0^{x+y} z^{-2\alpha} b(z,x,y)\ dz - 2 \left( \frac{x^{-2\alpha} + y^{-2\alpha}}{2} \right) \\
& \le \left[ E(x,y) + \beta_{-2\alpha} (1-E(x,y)) \right] (x+y)^{-2\alpha} - 2 \left( \frac{x+y}{2} \right)^{-2\alpha} \\
& = \left[ \beta_{-2\alpha} - (\beta_{-2\alpha}-1) E(x,y) - 2^{1+2\alpha} \right) (x+y)^{-2\alpha}
\end{align*}
for $(x,y)\in (0,\infty)^2$. Thanks to \eqref{p7}, we further obtain
\begin{equation*}
\zeta_{-2\alpha}(x,y)\le 0\,, \quad (x,y)\in (0,1)^2\,, \qquad \zeta_{-2\alpha}(x,y)\le \beta_{-2\alpha} (x+y)^{-2\alpha}\,, \quad (x,y)\in (0,\infty)^2\,,
\end{equation*}
and we infer from \eqref{p1}, \eqref{p9}, and the symmetry of $K_n$ and $\zeta_{-2\alpha}$ that, for $t\in (0,T_n)$, 
\begin{align*}
\frac{d}{dt} M_{-2\alpha}(f_n(t)) & = \frac{1}{2} \int_0^1 \int_0^1 \zeta_{-2\alpha}(x,y) K_n(x,y) f_n(t,x) f_n(t,y)\ dydx \\
& \qquad + \int_0^1 \int_1^\infty \zeta_{-2\alpha}(x,y) K_n(x,y) f_n(t,x) f_n(t,y)\ dydx \\
& \qquad +  \frac{1}{2} \int_1^\infty \int_1^\infty \zeta_{-2\alpha}(x,y) K_n(x,y) f_n(t,x) f_n(t,y)\ dydx \\
& \le k_1 \beta_{-2\alpha} \int_0^1 \int_1^\infty (x+y)^{-2\alpha} x^{-\alpha} y f_n(t,x) f_n(,y)\ dydx \\
& \qquad + \frac{k_1 \beta_{-2\alpha}}{2} \int_1^\infty \int_1^\infty (x+y)^{-2\alpha}x y f_n(t,x) f_n(,y)\ dydx \,.
\end{align*}
We now observe that, since $-2\alpha < -\alpha<0$,
\begin{align*}
(x+y)^{-2\alpha} x^{-\alpha} y & \le (x+y)^{-\alpha} x^{-\alpha} y \le x^{-2\alpha} y\,, \qquad (x,y)\in (0,1)\times (1,\infty)\,, \\
(x+y)^{-2\alpha} x y & \le xy\,, \qquad (x,y)\in (1,\infty)^2\,,
\end{align*}
so that, using also \eqref{p21},
\begin{equation*}
\frac{d}{dt} M_{-2\alpha}(f_n(t)) \le k_1 \beta_{-2\alpha} \varrho \left( \varrho +  M_{-2\alpha}(f_n(t)) \right)\,.
\end{equation*}
Hence, after integration with respect to time,  
\begin{align}
M_{-2\alpha}(f_n(t)) & \le M_{-2\alpha}(f_n^{in}) e^{k_1 \beta_{-2\alpha} \varrho t} + \varrho \left( e^{k_1 \beta_{-2\alpha}\varrho t} - 1 \right) \nonumber \\
& \le \left( M_{-2\alpha}(f^{in}) + \varrho \right) e^{k_1 \beta_{-2\alpha}\varrho t} \,, \qquad t\in [0,T_n)\,. \label{p26}
\end{align}

Now, for $t\in [0,T_n)$, it follows from H\"older's inequality and \eqref{p21} that 
\begin{equation*}
M_0(f_n(t)) \le M_{-2\alpha}(f_n(t))^{1/(1+2\alpha)} M_1(f_n(t))^{2\alpha/(1+2\alpha)} \le \varrho^{2\alpha/(1+2\alpha)} M_{-2\alpha}(f_n(t))^{1/(1+2\alpha)}\,,
\end{equation*}
and a first consequence of \eqref{p26} is that $M_0(f_n)$ cannot blow up in finite time, which excludes the occurrence of \eqref{p200} and thus implies that $T_n=\infty$. Next, the estimate \eqref{p26} readily gives \eqref{p25a}, while we deduce from H\"older's inequality, \eqref{p21}, and \eqref{p26} that, for $t\in (0,\infty)$,
\begin{align*}
M_{-\alpha}(f_n(t)) & \le M_{-2\alpha}(f_n(t))^{(1+\alpha)/(1+2\alpha)} M_1(f_n(t))^{\alpha/(1+2\alpha)} \\
& \le \varrho^{\alpha/(1+2\alpha)} \left( M_{-2\alpha}(f^{in}) + \varrho \right)^{(1+\alpha)/(1+2\alpha)} e^{k_1 (1+\alpha)\beta_{-2\alpha}\varrho t/(1+2\alpha)}\,,
\end{align*} 
and thereby complete the proof.
\end{proof}

We now turn to an estimate for superlinear moments when the growth of $K$ is at most sublinear.

\begin{lemma}\label{lemb4}
\refstepcounter{NCEx}\label{cst2} Assume further that $K$ satisfies \eqref{p2} and there is a convex function $\varphi\in C^2([0,\infty))$ such that
\begin{equation}
M_\varphi(f^{in}) = \int_0^\infty \varphi(x) f^{in}(x)\ dx < \infty \,, \label{pj3}
\end{equation}
and $\varphi$ has the following properties: $\varphi(0)=\varphi'(0)=0$, $\varphi'$ is a concave function which is positive in $(0,\infty)$. Then, for all $T>0$, there is $C_{\ref{cst2}}(T,\varphi)>0$ such that 
\begin{equation}
M_{\varphi}(f_n(t)) \le C_{\ref{cst2}}(T,\varphi)\,, \qquad t\in [0,T]\,, \ n\ge 1\,. \label{p27}
\end{equation}
\end{lemma}

\begin{proof}
We set $\varphi_1(x) := \varphi(x)/x$ for $x\in (0,\infty)$, $\varphi_1(0)=0$ and first recall that the properties of $\varphi$ entails that $\varphi_1$ is concave and non-decreasing \cite[Proposition~7.1.9~(c)]{BLL19}. It then follows from Jensen's inequality and \eqref{p40} that, for $(x,y)\in (0,\infty)^2$,
\begin{align*}
\int_0^{x+y} \varphi(z) b(z,x,y)\ dz & = (x+y) \int_0^{x+y} \varphi_1(z) \frac{z b(z,x,y)}{x+y}\ dz \\
& \le (x+y) \varphi_1 \left( \int_0^{x+y} z \frac{z b(z,x,y)}{x+y}\ dz \right) \\
& \le (x+y) \varphi_1 \left( \int_0^{x+y} z b(z,x,y)\ dz \right) \\
& = (x+y) \varphi_1(x+y) = \varphi(x+y)\,.
\end{align*}
Consequently, since $E\le 1$ by \eqref{p7a},
\begin{align}
\zeta_\varphi(x,y) & \le E(x,y) \varphi(x+y) + (1-E(x,y)) \varphi(x+y) -\varphi(x) - \varphi(y) \nonumber \\
& = \varphi(x+y) -\varphi(x) - \varphi(y)\,. \label{p28}
\end{align}
Owing to the concavity of $\varphi'$, 
\begin{equation*}
\varphi''(z)\le \varphi''(0) \;\;\text{ and }\;\; \varphi'(2z) \le 2 \varphi'(z)\,, \qquad z\in (0,\infty)\,,
\end{equation*} 
so that, for $(x,y)\in (0,\infty)^2$,
\begin{equation}
\varphi(x+y) -\varphi(x) - \varphi(y) = \int_0^x \int_0^y \varphi''(x_*+y_*)\ dy_*dx_* \le \varphi''(0) x y \label{p29}
\end{equation}
and,  for $0<x\le y$,
\begin{align}
\varphi(x+y) -\varphi(x) - \varphi(y) & \le \int_{y}^{x+y} \varphi'(z)\ dz \le x \varphi'(x+y) \nonumber \\
& \le x \varphi'(2y) \le 2 x \varphi'(y) \le 4 x \frac{\varphi(y)}{y}\,, \label{p210}
\end{align}
the last inequality being a consequence of the property $y\varphi'(y) \le 2 \varphi(y)$, see \cite[Proposition~7.1.9~(a)]{BLL19}. Also, by \cite[Proposition~7.1.9~(e)]{BLL19}
\begin{equation}
(x+y) \left[ \varphi(x+y) -\varphi(x) - \varphi(y) \right] \le 2 \left[ x \varphi(y) + y \varphi(x) \right]\,, \qquad (x,y)\in (0,\infty)^2\,. \label{p211}
\end{equation}
We then infer from \eqref{p1}, \eqref{p2}, \eqref{p9}, \eqref{p28}, \eqref{p29}, \eqref{p210}, and \eqref{p211} that, for $t>0$, 
\begin{align*}
\frac{d}{dt} M_\varphi(f_n(t)) & \le \frac{1}{2} \int_0^\infty \int_0^\infty \left[ \varphi(x+y)-\varphi(x) -\varphi(y) \right] K_n(x,y) f_n(t,x) f_n(t,y)\ dydx \\
& \le \frac{\varphi''(0)}{2} \int_0^1 \int_0^1 xy K_n(x,y) f_n(t,x) f_n(t,y)\ dydx \\
& \qquad + 4 \int_0^1 \int_1^\infty x \frac{\varphi(y)}{y} K_n(x,y) f_n(t,x) f_n(t,y)\ dydx \\
& \qquad + \int_1^\infty \int_1^\infty \frac{x\varphi(y)+y\varphi(x)}{x+y} K_n(x,y) f_n(t,x) f_n(t,y)\ dydx \\
& \le \frac{k_1 \varphi''(0)}{2} M_{1-\alpha}(f_n(t))^2 + 4k_1 M_{1-\alpha}(f_n(t)) M_\varphi(f_n(t)) \\
& \qquad + 2k_2 M_1(f_n(t)) M_\varphi(f_n(t))\,.
\end{align*}
We next use \eqref{p21} to obtain
\begin{equation}
\frac{d}{dt} M_\varphi(f_n(t)) \le k_1 \varphi''(0) M_{1-\alpha}(f_n(t))^2 + 2 \left( 2k_1 M_{1-\alpha}(f_n(t)) + k_2 \varrho \right) M_\varphi(f_n(t)) \,. \label{p212}
\end{equation}

Now, let $T>0$ and $t\in (0,T)$. By H\"older's inequality, \eqref{p21}, and \eqref{p25a},
\begin{equation*}
M_{1-\alpha}(f_n(t)) \le M_1(f_n(t))^{(1+\alpha)/(1+2\alpha)} M_{-2\alpha}(f_n(t))^{\alpha/(1+2\alpha)} \le \varrho^{(1+\alpha)/(1+2\alpha)} C_{\ref{cst1}}(T)^{\alpha/(1+2\alpha)}\,,
\end{equation*}
and we deduce from \eqref{p212} that
\begin{equation*}
\frac{d}{dt} M_\varphi(f_n(t)) \le C(T) \left( 1 + M_\varphi(f_n(t)) \right)\,, \qquad t\in (0,T)\,.
\end{equation*}
Integrating the above differential inequality and using \eqref{p20b} and \eqref{pj3} give \eqref{p27}.
\end{proof}

\begin{corollary}\label{corb4.5}
\refstepcounter{NCEx}\label{cst5} Assume further that $K$ satisfies \eqref{p2} and that $f^{in}\in X_2$. Then , for all $T>0$, there is $C_{\ref{cst5}}(T)>0$ such that 
\begin{equation*}
M_2(f_n(t)) \le C_{\ref{cst5}}(T)\,, \qquad t\in [0,T]\,, \ n\ge 1\,. 
\end{equation*}
\end{corollary}

\begin{proof}
Since the function $V_2: x\mapsto x^2$ is convex on $[0,\infty)$ with concave derivative and satisfies $V_2(0)=V_2'(0) =0$, Corollary~\ref{corb4.5} is a straightforward consequence of Lemma~\ref{lemb4} with $\varphi=V_2$.
\end{proof}

The next step is the uniform integrability in $X_{-\alpha}$. 

\begin{lemma}\label{lemb5}
For any $T>0$, there is a sequentially weakly compact subset $\mathcal{K}(T)$ of $X_{-\alpha}$ such that
\begin{equation*}
f_n(t) \in \mathcal{K}(T)\,, \qquad t\in [0,T]\,, \ n\ge 1\,.
\end{equation*}
\end{lemma}

\begin{proof}
We argue as in \cite{Stew89}, the singularity of $K$ for small sizes being handled as in \cite{BaGib, CCGW15, CCWa15}. Let $T>0$, $R>1$, and $\delta\in (0,1)$, and define
\begin{equation*}
\mathcal{E}_{n,R}(t,\delta) := \sup\left\{ \int_A x^{-\alpha} f_n(t,x)\ dx\ :\ A\subset (0,R)\,, \ |A|\le \delta \right\}
\end{equation*} 
for $t\in [0,T]$. We also set 
\begin{equation*}
\mathcal{E}^{in}(\delta) := \sup\left\{ \int_A x^{-\alpha} f^{in}(x)\ dx\ :\ A\subset (0,\infty)\,, \ |A|\le \delta \right\}\,,
\end{equation*} 
and note that
\begin{equation}
\mathcal{E}_{n,R}(0,\delta) \le \mathcal{E}^{in}(\delta)\,, \label{p213} 
\end{equation}
while the integrability properties of $f^{in}$ ensure that
\begin{equation}
\lim_{\delta\to 0} \mathcal{E}^{in}(\delta) = 0 \,. \label{p214}
\end{equation}

Consider now a measurable subset $A\subset (0,R)$ with finite measure $|A|\le \delta$ and set $\phi_A(x) := x^{-\alpha} \mathbf{1}_A(x)$ for $x\in (0,\infty)$. By \eqref{p5a} and \eqref{p7a}, 
\begin{align*}
\zeta_{\phi_A}(x,y) & \le E(x,y) (x+y)^{-\alpha} \mathbf{1}_{A}(x+y) + (1-E(x,y)) \omega(|A|) (x+y)^{-\alpha} \left( x^{-\theta} + y^{-\theta} \right) \\
& \le (x+y)^{-\alpha} \mathbf{1}_{A}(x+y) + \omega(\delta) (x+y)^{-\alpha} \left( x^{-\theta} + y^{-\theta} \right)
\end{align*}
for $(x,y)\in (0,\infty)^2$, and it follows from \eqref{p9} that, for $t\in (0,T)$,
\begin{equation}
\begin{split}
& \frac{d}{dt} \int_0^\infty \phi_A(x) f_n(t,x))\ dx \\
& \qquad \le \frac{1}{2} \int_0^\infty \int_0^\infty (x+y)^{-\alpha} \mathbf{1}_A(x+y) K_n(x,y) f_n(t,x) f_n(t,y)\ dydx \\
& \qquad\qquad + \frac{\omega(\delta)}{2} \int_0^\infty \int_0^\infty (x+y)^{-\alpha} \left( x^{-\theta} + y^{-\theta} \right) K_n(x,y) f_n(t,x) f_n(t,y)\ dydx\,.
\end{split} \label{p215}
\end{equation}

On the one hand, since $A\subset (0,R)$, we infer from \eqref{p1} that
\begin{align*}
I_n(t) & := \frac{1}{2} \int_0^\infty \int_0^\infty (x+y)^{-\alpha} \mathbf{1}_A(x+y) K_n(x,y) f_n(t,x) f_n(t,y)\ dydx \\
& = \frac{1}{2} \int_0^R \int_0^R (x+y)^{-\alpha} \mathbf{1}_A(x+y) K_n(x,y) f_n(t,x) f_n(t,y)\ dydx \\
& \le \frac{k_1}{2} \int_0^1 \int_0^1 (x+y)^{-\alpha} (xy)^{-\alpha} \mathbf{1}_A(x+y) f_n(t,x) f_n(t,y)\ dydx \\
& \qquad + k_1 \int_0^1 \int_1^R (x+y)^{-\alpha} x^{-\alpha} y \mathbf{1}_A(x+y) f_n(t,x) f_n(t,y)\ dydx \\
& \qquad + \frac{k_1}{2} \int_1^R \int_1^R (x+y)^{-\alpha} xy \mathbf{1}_A(x+y) f_n(t,x) f_n(t,y)\ dydx \\
& \le \frac{k_1}{2} \int_0^1 y^{-2\alpha} f_n(t,y) \int_0^1 x^{-\alpha} \mathbf{1}_{-y+A}(x) f_n(t,x)\ dx dy \\
& \qquad + k_1 \int_1^R y f_n(t,y) \int_0^1 x^{-\alpha} \mathbf{1}_{-y+A}(x) f_n(t,x)\ dx dy \\
& \qquad + \frac{k_1 R}{2} \int_1^R y f_n(t,y) \int_1^R x^{-\alpha} \mathbf{1}_{-y+A}(x) f_n(t,x)\ dx dy\,.
\end{align*}
Since $-y+A \subset (0,R)$ and $|-y+A|=|A|\le \delta$, we infer from \eqref{p21}, \eqref{p25a}, and the above inequality that 
\begin{equation}
I_n(t)  \le \left( \frac{k_1}{2} M_{-2\alpha}(f_n(t)) + k_1 \varrho + \frac{k_1 R \varrho}{2} \right) \mathcal{E}_{n,R}(t,\delta) \le C(T) R \mathcal{E}_{n,R}(t,\delta)\,. \label{p216}
\end{equation}
On the other hand, by \eqref{p1} and \eqref{p6b}, 
\begin{align*}
J_n(t) & := \frac{1}{2} \int_0^\infty \int_0^\infty (x+y)^{-\alpha} \left( x^{-\theta} + y^{-\theta} \right) K_n(x,y) f_n(t,x) f_n(t,y)\ dydx \\
& \le \frac{k_1}{2} \int_0^1 \int_0^1 (x+y)^{-\alpha} \left( x^{-\theta} + y^{-\theta} \right) (xy)^{-\alpha} f_n(t,x) f_n(t,y)\ dydx \\
& \qquad + k_1 \int_0^1 \int_1^\infty (x+y)^{-\alpha} \left( x^{-\theta} + y^{-\theta} \right) x^{-\alpha} y f_n(t,x) f_n(t,y)\ dydx \\
& \qquad + \frac{k_1}{2} \int_1^\infty \int_1^\infty (x+y)^{-\alpha} \left( x^{-\theta} + y^{-\theta} \right) xy f_n(t,x) f_n(t,y)\ dydx \\
& \le k_1 \int_0^1 \int_0^1 x^{-\theta-\alpha} y^{-2\alpha} f_n(t,x) f_n(t,y)\ dydx + 2k_1 \int_0^1  \int_1^\infty x^{-\theta-\alpha} y f_n(t,x) f_n(t,y)\ dydx \\
& \qquad + k_1 \int_1^\infty \int_1^\infty xy f_n(t,x) f_n(t,y)\ dydx \\
& \le k_1 M_{-2\alpha}(f_n(t))^2 + 2k_1 M_{-2\alpha}(f_n(t)) M_1(f_n(t)) + k_1 M_1(f_n(t))^2\,. 
\end{align*}
It then follows from \eqref{p21} and \eqref{p25a} that
\begin{equation}
J_n(t) \le C(T)\,. \label{p217}
\end{equation}

Gathering \eqref{p215}, \eqref{p216}, and \eqref{p217} leads us to
\begin{equation*}
\frac{d}{dt} \int_0^\infty \phi_A(x) f_n(t,x)\ dx \le C(T) \left[ \omega(\delta) + R \mathcal{E}_{n,R}(t,\delta) \right]\,, \qquad t\in (0,T) \,.
\end{equation*}
After integrating with respect to time and taking the supremum over all measurable subsets $A\subset (0,R)$ with finite measure $|A|\le \delta$, we find
\begin{equation*}
\mathcal{E}_{n,R}(t,\delta) \le \mathcal{E}_{n,R}(0,\delta) + C(T) \int_0^t \left[ \omega(\delta) + R \mathcal{E}_{n,R}(s,\delta) \right]\ ds\,, \qquad t\in [0,T]\,.
\end{equation*}
Therefore, by Gronwall's lemma and \eqref{p213}, \refstepcounter{NCEx}\label{cst3}
\begin{equation}
\mathcal{E}_{n,R}(t,\delta) \le \left[ \mathcal{E}_{n,R}(0,\delta) + C(T) \omega(\delta) \right] e^{C(T) R t} \le C_{\ref{cst3}}(T) \left[ \mathcal{E}^{in}(\delta) + \omega(\delta) \right] \,, \qquad t\in [0,T]\,. \label{p218}
\end{equation} 

We next define 
\begin{equation*}
\mathcal{E}(\delta) := \sup\left\{ \int_A x^{-\alpha} f_n(t,x)\ dx\ :\ A\subset (0,\infty)\,, \ |A|\le \delta\,, \ t\in [0,T]\,, \ n\ge 1 \right\}\,.
\end{equation*} 
Let $n\ge 1$, $T>0$, $t\in [0,T]$, and $\delta\in (0,1)$. If $A$ is a measurable subset in $(0,\infty)$ with finite measure $|A|\le\delta$ and $R>1$, then $A\cap (0,R)$ is a measurable subset of $(0,R)$ with finite measure $|A\cap (0,R)|\le |A|\le \delta$ and we infer from \eqref{p21} and \eqref{p218} that
\begin{align*}
\int_A x^{-\alpha} f_n(t,x)\ dx & \le \int_{A\cap (0,R)} x^{-\alpha} f_n(t,x)\ dx + \int_R^\infty x^{-\alpha} f_n(t,x)\ dx \\
& \le \mathcal{E}_{n,R}(t,\delta) + R^{-1-\alpha} \int_R^\infty x f_n(t,x)\ dx \\
& \le C_{\ref{cst3}}(T) \left( \mathcal{E}^{in}(\delta) + \omega(\delta) \right) + \frac{\varrho}{R^{1+\alpha}}\,.
\end{align*}
Since the right-hand side of the above inequality does not depend on $n\ge 1$, $t\in [0,T]$, and $A\subset (0,\infty)$ with finite measure $|A|\le\delta$, we readily conclude that
\begin{equation*}
\mathcal{E}(\delta) \le C_{\ref{cst3}}(T) \left[ \mathcal{E}^{in}(\delta) + \omega(\delta) \right] + \frac{\varrho}{R^{1+\alpha}}\,.
\end{equation*}
Owing to \eqref{p5b} and \eqref{p214}, we may let $\delta\to 0$ in the previous inequality to obtain
\begin{equation*}
\limsup_{\delta\to 0} \mathcal{E}(\delta) \le \frac{\varrho}{R^{1+\alpha}}\,.
\end{equation*}
The above inequality being valid for all $R>1$, we may let $R\to\infty$ and deduce from the non-negativity of $\mathcal{E}(\delta)$ that
\begin{equation*}
\lim_{\delta\to 0} \mathcal{E}(\delta) = 0\,.
\end{equation*}
This property guarantees that the family $\mathcal{F} := \left\{ f_n(t)\ :\ t\in [0,T]\,,\ n\ge 1\right\}$ is uniformly integrable in $X_{-\alpha}$. Since $\mathcal{F}$ is also bounded in $X_1$ by \eqref{p21}, we conclude from these two properties that $\mathcal{F}$ is relatively sequentially weakly compact in $X_{-\alpha}$ according to the Dunford-Pettis theorem.
\end{proof}

The last estimate to be derived is the time equicontinuity of the sequence $(f_n)_{n\ge 1}$.

\begin{lemma}\label{lemb6}
\refstepcounter{NCEx}\label{cst4} For any $T>0$, there is $C_{\ref{cst4}}(T)>0$ such that
\begin{equation*}
\int_0^\infty x^{-\alpha} |f_n(t,x) - f_n(s,x)|\ dx \le C_{\ref{cst4}}(T) |t-s|\,, \qquad (t,s)\in [0,T]^2\,,\ n\ge 1\,.
\end{equation*}
\end{lemma}

\begin{proof}
Let $t\in [0,T]$. By \eqref{p1}, \eqref{p21}, \eqref{p25a}, and \eqref{p25b}, 
\begin{align*}
\int_0^\infty x^{-\alpha} \mathcal{D}_n(f_n)(t,x)\ dx & \le k_1 \int_0^1\int_0^1 x^{-2\alpha} y^{-\alpha} f_n(t,x) f_n(t,y)\ dxdy \\
& \qquad + 2k_1 \int_0^1\int_1^\infty x^{-2\alpha} y f_n(t,x) f_n(t,y)\ dxdy \\
& \qquad + k_1 \int_1^\infty \int_1^\infty x^{1-\alpha} y f_n(t,x) f_n(t,y)\ dxdy \\
& \le k_1 M_{-2\alpha}(f_n(t)) M_{-\alpha}(f_n(t)) + 2k_1 M_{-2\alpha}(f_n(t)) M_1(f_n(t)) \\
& \qquad + k_1 M_1(f_n(t))^2 \\
& \le k_1 C_{\ref{cst1}}(T)^2 + 2 k_1 \varrho C_{\ref{cst1}}(T) + k_1 \varrho^2\,.
\end{align*} 
Next, using Fubini's theorem,
\begin{align*}
\int_0^\infty x^{-\alpha} \mathcal{B}_{c,n}(f_n)(t,x)\ dx & \le \frac{1}{2} \int_0^\infty \int_0^\infty (x+y)^{-\alpha} K_n(x,y) f_n(t,x) f_n(t,y)\ dydx \\
& \le \int_0^\infty x^{-\alpha} \mathcal{D}_n(f_n)(t,x)\ dx\,. 
\end{align*}
Finally, using again Fubini's theorem along with \eqref{p6c}, 
\begin{align*}
\int_0^\infty x^{-\alpha} \mathcal{B}_{b,n}(f_n)(t,x)\ dx & \le \frac{1}{2} \int_0^\infty \int_0^\infty \left( \int_0^{y+z} x^{-\alpha} b(x,y,z)\ dx \right) K_n(y,z) f_n(t,y) f_n(t,z)\ dzdy \\
& \le \frac{\beta_{-2\alpha}}{2} \int_0^\infty \int_0^\infty (y+z)^{-\alpha} K_n(y,z) f_n(t,y) f_n(t,z)\ dydz \\
& \le \beta_{-2\alpha} \int_0^\infty x^{-\alpha} \mathcal{D}_n(f_n)(t,x)\ dx\,. 
\end{align*}
Combining the above three estimates with \eqref{cecbt1} gives
\begin{align*}
\int_0^\infty x^{-\alpha} |\partial_t f_n(t,x)|\ dx & \le \left( 2 + \beta_{-2\alpha} \right) \int_0^\infty x^{-\alpha} \mathcal{D}_n(f_n)(t,x)\ dx \\
& \le k_1 \left( 2 + \beta_{-2\alpha} \right) \left( C_{\ref{cst1}}(T)+ \varrho \right)^2\,,
\end{align*}
from which Lemma~\ref{lemb6} readily follows.
\end{proof}

Gathering the outcome of Lemma~\ref{lemb5} and Lemma~\ref{lemb6} provides the relative compactness of the sequence $(f_n)_{n\ge 1}$ in suitable weighted $L^1$-spaces.

\begin{proposition}\label{propb7}
Let $T>0$. 
\begin{itemize}
\item[(a)] Assume further that $K$ satisfies \eqref{p3} and set $V(x) := \max\{ x^{-\alpha} , r(x)\}$ for $x\in (0,\infty)$. Then $(f_n)_{n\ge 1}$ is relatively compact in $C([0,T],X_{V,w})$.
\item[(b)] Assume further that $K$ and $f^{in}$ satisfy \eqref{p2} and \eqref{pj3}, respectively, where $\varphi\in C^2([0,\infty))$ is a convex function endowed with the following properties: $\varphi(0)=\varphi'(0)=0$, $\varphi'$ is a concave function which is positive in $(0,\infty)$, and 
\begin{equation}
\lim_{x\to\infty} \varphi'(x) = \lim_{x\to\infty} \frac{\varphi(x)}{x} = \infty\,. \label{pj4}
\end{equation}
Setting $V(x) := \max\{x^{-\alpha},x\}$ for $x\in (0,\infty)$, $(f_n)_{n\ge 1}$ is relatively compact in $C([0,T],X_{V,w})$.
\end{itemize}
\end{proposition}

\begin{proof} Let $t\in [0,T]$. According to Lemma~\ref{lemb6}, $(f_n(t))_{n\ge 1}$ is equicontinuous in $X_{-\alpha}$ for its norm-topology and is thus also weakly equicontinuous in $X_{-\alpha,w}$. Consequently, due to Lemma~\ref{lemb5}, we are in a position to apply a variant of the Arzel\`a-Ascoli theorem, see \cite[Theorem~A.3.1]{Vrab03}, to conclude that $(f_n)_{n\ge 1}$ is relatively compact in $C([0,T],X_{-\alpha,w})$; that is, there are $f\in C([0,T],X_{-\alpha,w})$ and a subsequence $(f_{n_j})_{j\ge 1}$ of $(f_n)_{n\ge 1}$ such that
\begin{equation}
\lim_{j\to\infty} \sup_{t\in [0,T]} \left| \int_0^\infty x^{-\alpha} (f_{n_j}-f)(t,x) \phi(x)\ dx \right| = 0 \label{p221}
\end{equation}
for all $\phi\in L^\infty(0,\infty)$.

\noindent\textbf{Case~(a):} For $t\in [0,T]$ and $R>1$, we deduce from \eqref{p21} and \eqref{p221} that
\begin{equation*}
\int_0^R x f(t,x)\ dx = \int_0^R x^{1+\alpha} x^{-\alpha} f(t,x)\ dx = \lim_{j\to\infty} \int_0^R x f_{n_j}(t,x)\ dx \le \varrho\,.
\end{equation*}
We then let $R\to\infty$ in the previous inequality and use Fatou's lemma to conclude that \begin{equation}
f(t)\in X_1 \;\;\text{ with }\;\; M_1(f(t))\le \varrho\,, \qquad t\in [0,T]\,. \label{p222}
\end{equation} 

Consider now $\phi\in L^\infty(0,\infty)$, $t\in [0,T]$, $j\ge 1$, and $R>1$. Then, by \eqref{p21} and \eqref{p222}, 
\begin{align*}
\int_R^\infty V(x) (f_{n_j}+f)(t,x)\ dx & \le \left( \frac{1}{R^{1+\alpha}} + \sup_{x>R}\left\{ \frac{r(x)}{x} \right\} \right) \int_R^\infty x (f_{n_j}+f)(t,x)\ dx \\
& \le 2\varrho \left( \frac{1}{R^{1+\alpha}} + \sup_{x>R}\left\{ \frac{r(x)}{x} \right\} \right)\,,
\end{align*}
so that
\begin{align*}
\left| \int_0^\infty V(x) (f_{n_j}-f)(t,x) \phi(x)\ dx \right| & \le \left| \int_0^R V(x) (f_{n_j}-f)(t,x) \phi(x)\ dx \right| \\
& \qquad +  \|\phi\|_{L^\infty(0,\infty)} \int_R^\infty V(x) (f_{n_j}+f)(t,x)\ dx  \\
& \le \left| \int_0^R x^{-\alpha} x^\alpha V(x) (f_{n_j}-f)(t,x) \phi(x)\ dx \right| \\
& \qquad + 2 \varrho \left( \frac{1}{R^{1+\alpha}} + \sup_{x>R}\left\{ \frac{r(x)}{x} \right\} \right)  \|\phi\|_{L^\infty(0,\infty)} \,.
\end{align*}
Since $x\mapsto x^\alpha V(x) \mathbf{1}_{(0,R)}(x)$ belongs to $L^\infty(0,\infty)$, we infer from \eqref{p221} that
\begin{align*}
& \limsup_{j\to\infty} \sup_{t\in [0,T]} \left\{ \left| \int_0^\infty V(x) (f_{n_j}-f)(t,x) \phi(x)\ dx \right| \right\} \\
& \qquad\qquad\qquad \le 2 \varrho \left( \frac{1}{R^{1+\alpha}} + \sup_{x>R}\left\{ \frac{r(x)}{x} \right\} \right)  \|\phi\|_{L^\infty(0,\infty)} \,.
\end{align*}
We may let $R\to \infty$ in the above inequality and deduce from \eqref{p3b} that
\begin{equation*}
\lim_{j\to\infty} \sup_{t\in [0,T]} \left\{ \left| \int_0^\infty V(x) (f_{n_j}-f)(t,x) \phi(x)\ dx \right| \right\} = 0 \,,
\end{equation*}
and we have thus shown that $(f_{n_j})_{j\ge 1}$ converges to $f$ in $C([0,T],X_{V,w})$.

\noindent\textbf{Case~(b):} As in Case~(a), we deduce from \eqref{p27} and \eqref{p221} that, for $t\in [0,T]$ and $R>1$,
\begin{equation*}
\int_0^R \varphi(x) f(t,x)\ dx = \int_0^R x^{\alpha} \varphi(x) x^{-\alpha} f(t,x)\ dx = \lim_{j\to\infty} \int_0^R \varphi(x) f_{n_j}(t,x)\ dx \le C_{\ref{cst2}}(T,\varphi)\,.
\end{equation*}
We then let $R\to\infty$ in the previous inequality and use Fatou's lemma to conclude that \begin{equation}
f(t)\in X_\varphi \;\;\text{ with }\;\; M_\varphi(f(t))\le C_{\ref{cst2}}(T,\varphi)\,, \qquad t\in [0,T]\,. \label{p223}
\end{equation}
Consider now $\phi\in L^\infty(0,\infty)$, $t\in [0,T]$, $j\ge 1$, and $R>1$. We argue as in Case~(a) and infer from \eqref{p27} and \eqref{p223} that
\begin{align*}
\left| \int_0^\infty V(x) (f_{n_j}-f)(t,x) \phi(x)\ dx \right| & \le \left| \int_0^R V(x) (f_{n_j}-f)(t,x) \phi(x)\ dx \right| \\
& \qquad + 2 C_{\ref{cst2}}(T,\varphi) \sup_{x>R}\left\{ \frac{x}{\varphi(x)} \right\} \|\phi\|_{L^\infty(0,\infty)} \,.
\end{align*}
Again as in Case~(a), we may first pass to the limit $j\to\infty$ with the help of \eqref{p221} and afterwards let $R\to\infty$ with the help of \eqref{pj4}, so as to end up with
\begin{equation*}
\lim_{j\to\infty} \sup_{t\in [0,T]} \left\{ \left| \int_0^\infty V(x) (f_{n_j}-f)(t,x) \phi(x)\ dx \right| \right\} = 0 \,,
\end{equation*}
as claimed.
\end{proof}

\begin{proof}[Proof of Theorem~\ref{tha1}~(a)] Set $V(x) = \max\{x^{-\alpha} , r(x)\}$ for $x\in (0,\infty)$. By Proposition~\ref{propb7}~(a), the sequence $(f_n)_{n\ge 1}$ is relatively compact in $C([0,T],X_{V,w})$ for each $T>0$. A diagonal process then guarantees the existence of $f\in C([0,\infty),X_{V,w})$  and a subsequence of $(f_n)_{n\ge 1}$ (not relabeled) such that
\begin{equation}	
f_n \longrightarrow f \;\;\text{ in }\;\; C([0,T],X_{V,w}) \;\;\text{ for all }\;\; T>0\,. \label{p224}
\end{equation} 
A first consequence of \eqref{p21}, the non-negativity of each $f_n$, $n\ge 1$, and \eqref{p224} is that
\begin{equation*}
f(t)\in X_1^+ \;\;\text{ with }\;\; M_1(f(t))\le\varrho\,, \qquad t\ge 0\,, 
\end{equation*}
the proof of the upper bound of the first moment being the same as that of \eqref{p222}. Also, for $T>0$, $t\in [0,T]$, and $\varepsilon\in (0,1)$, we infer from \eqref{p25a} and \eqref{p224} that	
\begin{equation*}
\int_\varepsilon^\infty x^{-2\alpha} f(t,x)\ dx = \lim_{n\to\infty} \int_\varepsilon^\infty x^{-2\alpha} f_n(t,x)\ dx \le C_{\ref{cst1}}(T)\,.
\end{equation*}
Letting $\varepsilon\to 0$ in the previous inequality, we deduce from Fatou's lemma that $M_{-2\alpha}(f(t))\le C_{\ref{cst1}}(T)$. Consequently, $f$ satisfies \eqref{pa1} and \eqref{pa2}.	

It remains to check that $f$ satisfies the weak formulation \eqref{pa3} of \eqref{cecb}. To this end, consider $\phi\in L^\infty(0,\infty)$ and $t>0$. Since $r(x)\ge 1$ for all $x\in (0,\infty)$ by \eqref{p3}, there holds $1\le V$ and it easily follows from \eqref{p20b} and \eqref{p224} that
\begin{equation}
\lim_{n\to\infty} \int_0^\infty \phi(x) [f_n(t,x) - f_n^{in}(x)]\ dx = \int_0^\infty \phi(x) [f(t,x) - f^{in}]\ dx\,. \label{p230}
\end{equation}
We next integrate \eqref{p9} over $(0,t)$ and obtain
\begin{equation}
\int_0^\infty \phi(x) [f_n(t,x) - f_n^{in}(x)]\ dx = \frac{1}{2} \int_0^t \int_0^\infty \int_0^\infty \zeta_\phi(x,y) K_n(x,y) f_n(s,x) f_n(s,y)\ dydxds\,. \label{p226}
\end{equation}
On the one hand, it follows from \eqref{p224} that
\begin{equation}
F_n \longrightarrow F \;\;\text{ in }\;\; C([0,T],X_{0,w}\times X_{0,w}) \;\;\text{ for all }\;\; T>0\,, \label{p227}
\end{equation}
where
\begin{equation*}
F_n(t,x,y) := V(x) V(y) f_n(t,x) f_n(t,y) \;\;\text{ and }\;\; F(t,x,y) := V(x) V(y) f(t,x) f(t,y)
\end{equation*}
for $(t,x,y)\in [0,\infty)\times (0,\infty)^2$. 
On the other hand, since $K(x,y)\le (1+k_1) V(x) V(y)$ by \eqref{p1} and \eqref{p3a} and
\begin{equation*}
|\zeta_\phi(x,y)| \le \left( 3 + \beta_{-2\alpha} \right) \|\phi\|_{L^\infty(0,\infty)} 
\end{equation*}
by \eqref{p6d} and \eqref{p7a}, we see that 
\begin{equation}
\frac{|\zeta_\phi(x,y)| K_n(x,y)}{V(x) V(y)} \le (1+k_1) \left( 3 + \beta_{-2\alpha} \right) \|\phi\|_{L^\infty(0,\infty)} \,, \qquad (x,y)\in (0,\infty)^2\,, \ n\ge 1\,, \label{p228}
\end{equation}
while \eqref{p1} and \eqref{p20a} ensure that
\begin{equation}
\lim_{n\to\infty} \frac{\zeta_\phi(x,y) K_n(x,y)}{V(x) V(y)} = \frac{\zeta_\phi(x,y) K(x,y)}{V(x) V(y)} \,, \qquad (x,y)\in (0,\infty)^2\,. \label{p229}
\end{equation}
Owing to \eqref{p227}, \eqref{p228}, and \eqref{p229}, we may invoke \cite[Proposition~2.61]{FoLe07} to conclude that 
\begin{align*}
& \lim_{n\to \infty} \int_0^t \int_0^\infty \int_0^\infty \frac{\zeta_\phi(x,y) K_n(x,y)}{V(x) V(y)} F_n(s,x,y)\ dxdyds \\
& \hspace{2cm} = \int_0^t \int_0^\infty \int_0^\infty \frac{\zeta_\phi(x,y) K(x,y)}{V(x) V(y)} F(s,x,y)\ dxdyds\,.
\end{align*}
Equivalently,
\begin{align*}
& \lim_{n\to \infty} \int_0^t \int_0^\infty \int_0^\infty \zeta_\phi(x,y) K_n(x,y) f_n(s,x) f_n(s,y)\ dxdyds \\
& \hspace{2cm} = \int_0^t \int_0^\infty \int_0^\infty \zeta_\phi(x,y) K(x,y) f(s,x) f(s,y)\ dxdyds\,.
\end{align*}
Combining the above identity with \eqref{p230} and \eqref{p226} ensures that $f$ satisfies \eqref{pa3} and completes the proof.
\end{proof}

\begin{proof}[Proof of Theorem~\ref{tha1}~(b)]
We first recall that, since $ x\mapsto x\in L^1((0,\infty),f^{in}(x) dx)$, a refined version of the de la Vall\'ee-Poussin theorem, see \cite{Le77} and \cite[Theorem~7.1.6]{BLL19}, ensures that there is a convex function $\psi\in C^2([0,\infty))$ such that
\begin{equation}
M_\psi(f^{in}) = \int_0^\infty \psi(x) f^{in}(x)\ dx < \infty \,, \label{pi4}
\end{equation}
and $\psi$ has the following properties: $\psi(0)=\psi'(0)=0$, $\psi'$ is a concave function which is positive in $(0,\infty)$, and 
\begin{equation}
\lim_{x\to\infty} \psi'(x) = \lim_{x\to\infty} \frac{\psi(x)}{x} = \infty\,. \label{pi5}
\end{equation}
We set $V(x)=\max\{x^{-\alpha} , x\}$ for $x\in (0,\infty)$. Thanks to \eqref{pi4} and \eqref{pi5}, the assumptions required to apply Proposition~\ref{propb7} are satisfied with $\varphi=\psi$ and we infer from Proposition~\ref{propb7} that
\begin{equation}	
f_n \longrightarrow f \;\;\text{ in }\;\; C([0,T],X_{V,w}) \;\;\text{ for all }\;\; T>0\,. \label{p231}
\end{equation} 
Since $V(x)\ge 1$ for $x\in (0,\infty)$, a first consequence of \eqref{p231} is that
\begin{equation*}
f_n \longrightarrow f \;\;\text{ in }\;\; C([0,T],X_{0,w}\cap X_{1,w}) \;\;\text{ for all }\;\; T>0\,.
\end{equation*}
Together with the non-negativity of each $f_n$, $n\ge 1$, and \eqref{p21}, these convergences imply that $f$ satisfies \eqref{pa4} and \eqref{pa5}, since
\begin{equation*}
M_1(f(t)) = \lim_{n\to\infty} M_1(f_n(t)) = \lim_{n\to\infty} M_1(f_n^{in}) = M_1(f^{in})\,, \qquad t\ge 0\,.
\end{equation*}

We next notice that \eqref{p1} entails that $K(x,y)\le k_1 V(x) V(y)$ for $(x,y)\in (0,\infty)^2$ and proceed as in the proof of Theorem~\ref{tha1}~(a) to show that $f$ satisfies \eqref{pa3}, which completes the proof.
\end{proof}

\begin{proof}[Proof of Theorem~\ref{tha1}~(c)]
The proof of Theorem~\ref{tha1}~(c) is the same as that of Theorem~\ref{tha1}~(b), the additional bound in $X_2$ being derived with the help of Corollary~\ref{corb4.5} and the convergence \eqref{p231}.
\end{proof}

\subsection{Existence: $\alpha=0$}\label{sec2.2}

In this section, besides \eqref{p1}, \eqref{p40}, and \eqref{p7}, we further assume that $\alpha=0$, while the daughter distribution function $b$ satisfies \eqref{p5} and \eqref{p4}. We also assume that, besides \eqref{pi1},
\begin{equation}
f^{in}\in X_{-\theta}\,,\label{pi3}
\end{equation}
recalling that $\theta$ is defined in \eqref{p5a}. We begin with an estimate on $f_n$ in $X_0$ and $X_{-\theta}$.

\begin{lemma}\label{lemb2}
For all $n\ge 1$, $T_n=\infty$ and, for any $T>0$, there is $C_{\ref{cst1}}(T)>0$ such that 
	\begin{align}
	M_0(f_n(t)) & \le C_{\ref{cst1}}(T)\,, \qquad t\in [0,T]\,,\ n\ge 1\,, \label{p23} \\
	M_{-\theta}(f_n(t)) & \le C_{\ref{cst1}}(T)\,, \qquad t\in [0,T]\,,\ n\ge 1\,. \label{p24}
	\end{align}
\end{lemma}

\begin{proof}
 By \eqref{p4a}, 
\begin{equation*}
\zeta_0(x,y) \le E(x,y) + \beta_0 (1-E(x,y))- 2 = \beta_0- 2 -(\beta_0-1) E(x,y) 
\end{equation*}
for $(x,y)\in (0,\infty)^2$. Consequently, by \eqref{p7}, 
\begin{equation*}
\zeta_0(x,y)\le 0\,, \quad (x,y)\in (0,1)^2\,, \qquad \zeta_0(x,y) \le \beta_0\,, \quad (x,y)\in (0,\infty)^2\,,
\end{equation*}
and we infer from \eqref{p1}, \eqref{p9}, and the symmetry of $K_n$ and $\zeta_0$ that, for $t\in (0,T_n)$,
\begin{align*}
\frac{d}{dt} M_0(f_n(t)) & = \frac{1}{2} \int_0^1 \int_0^1 \zeta_0(x,y) K_n(x,y) f_n(t,x) f_n(t,y)\ dydx \\
& \qquad + \int_0^1 \int_1^\infty \zeta_0(x,y) K_n(x,y) f_n(t,x) f_n(t,y)\ dydx \\
& \qquad +  \frac{1}{2} \int_1^\infty \int_1^\infty \zeta_0(x,y) K_n(x,y) f_n(t,x) f_n(t,y)\ dydx \\
& \le k_1 \beta_0 \int_0^1 \int_1^\infty y f_n(t,x) f_n(t,y)\ dydx \\
& \qquad + \frac{k_1 \beta_0}{2} \int_1^\infty \int_1^\infty x y f_n(t,x) f_n(t,y)\ dydx \\
& \le k_1 \beta_0 M_1(f_n(t)) M_0(f_n(t)) + k_1 \beta_0 M_1(f_n(t))^2\,.
\end{align*}
Recalling \eqref{p21}, we deduce that
\begin{equation*}
\frac{d}{dt} M_0(f_n(t)) \le k_1 \beta_0 \varrho \left[\varrho +M_0(f_n(t)) \right]\,.
\end{equation*}
Hence, integrating with respect to time and using \eqref{p20b}, 
\begin{equation}
M_0(f_n(t)) \le M_0(f_n^{in}) e^{k_1 \beta_0\varrho t} + \varrho \left( e^{k_1 \beta_0\varrho t}- 1 \right) \le (\varrho + M_0(f^{in})) e^{k_1 \beta_0\varrho t} \label{p201}
\end{equation}
for $t\in [0,T_n)$. On the one hand, \eqref{p201} prevents the occurrence of \eqref{p200} and thereby implies that $T_n=\infty$, according to Proposition~\ref{propb1}. On the other hand, \eqref{p23} is a straightforward consequence of \eqref{p201}.

It next follows from \eqref{p4b} and the convexity of $x\mapsto x^{-\theta}$ that, for $t\ge 0$,
\begin{align*}
\zeta_{-\theta}(x,y) & = 2^{-\theta} E(x,y) \left( \frac{x+y}{2} \right)^{-\theta} + (1-E(x,y)) \int_0^{x+y} z^{-\theta} b(z,x,y)\ dz - x^{-\theta} - y^{-\theta} \\
& \le 2^{-\theta-1} E(x,y) \left( x^{-\theta} + y^{-\theta} \right) + \frac{\beta_{-\theta}}{2} (1-E(x,y)) \left( x^{-\theta} + y^{-\theta} \right) - x^{-\theta} - y^{-\theta} \\
& \le  \frac{\beta_{-\theta}}{2} \left( x^{-\theta} + y^{-\theta} \right) \le \beta_{-\theta} \left( x^{-\theta} + y^{-\theta} \right)\,.
\end{align*}
Consequently, by \eqref{p1}, \eqref{p21}, and \eqref{p9}, 
\begin{align*}
\frac{d}{dt} M_{-\theta}(f_n(t)) & \le  \frac{\beta_{-\theta}}{2} \int_0^\infty \int_0^\infty \left( x^{-\theta} + y^{-\theta} \right) K(x,y) f_n(t,x) f_n(t,y)\ dydx \\
& = \beta_{-\theta} \int_0^\infty \int_0^\infty x^{-\theta} K(x,y) f_n(t,x) f_n(t,y)\ dydx \\
& \le k_1 \beta_{-\theta} \int_0^1 \int_0^1 x^{-\theta} f_n(t,x) f_n(t,y)\ dydx \\
& \qquad + k_1 \beta_{-\theta} \int_0^1 \int_1^\infty x^{-\theta} y f_n(t,x) f_n(t,y)\ dydx \\
& \qquad + k_1 \beta_{-\theta} \int_1^\infty \int_0^1 x^{1-\theta} f_n(t,x) f_n(t,y)\ dydx \\
& \qquad + k_1 \beta_{-\theta} \int_1^\infty \int_1^\infty x^{1-\theta} y f_n(t,x) f_n(t,y)\ dydx \\
& \le k_1 \beta_{-\theta} \left( M_0(f_n(t)) M_{-\theta}(f_n(t)) + M_1(f_n(t)) M_{-\theta}(f_n(t)) \right) \\
& \qquad + k_1 \beta_{-\theta} \left( M_0(f_n(t)) M_1(f_n(t)) + M_1(f_n(t))^2 \right) \\
& \le k_1 \beta_{-\theta} \left[ \varrho + M_0(f_n(t)) \right] \left[ \varrho + M_{-\theta}(f_n(t)) \right]\,.
\end{align*}
Hence, after integration with respect to time,
\begin{align*}
M_{-\theta}(f_n(t)) & \le M_{-\theta}(f_n^{in}) \exp\left\{ k_1 \beta_{-\theta} \int_0^t \left[ \varrho + M_0(f_n(s)) \right]\ ds \right\} \\
& \qquad + \varrho \left( \exp\left\{ k_1 \beta_{-\theta} \int_0^t \left[ \varrho + M_0(f_n(s)) \right]\ ds \right\} - 1 \right)\,.
\end{align*}
Since $M_{-\theta}(f_n^{in}) \le M_{-\theta}(f^{in})<\infty$ by \eqref{p20b} and \eqref{pi3}, \eqref{p24} readily follows from \eqref{p23} and the above inequality.
\end{proof}

\begin{proof}[Proof of Theorem~\ref{tha2}] A careful inspection of the proofs of Lemma~\ref{lemb4}, Corollary~\ref{corb4.5}, Lemma~\ref{lemb5}, Lemma~\ref{lemb6}, and Proposition~\ref{propb7} reveals that these results are still valid for $\alpha=0$, every recourse to Lemma~\ref{lemb3} being replaced by Lemma~\ref{lemb2}. The remainder of the proof of Theorem~\ref{tha2} is then the same as that of Theorem~\ref{tha1}, replacing again $\alpha$ by zero and Lemma~\ref{lemb3} by Lemma~\ref{lemb2}.
\end{proof}

\subsection{Existence: $\alpha=0$}\label{sec2.4}

In this section, besides \eqref{p1}, \eqref{p40}, and \eqref{p7}, we further assume that $\alpha=0$, while the daughter distribution function $b$ satisfies \eqref{p5}, \eqref{p4a}, and, either \eqref{paa4a}, or \eqref{paa4}. Since $f^{in}\in X_0^+$ by \eqref{pi1}, it follows from Lemma~\ref{leap1} that there is a non-negative convex and non-increasing function $\Phi$ such that
\begin{equation}
\int_0^\infty \Phi(x) f^{in}(x)\ dx < \infty\,,\label{pv1}
\end{equation}
and
\begin{equation}
\lim_{x\to 0} \Phi(x) =  \infty\,, \qquad \lim_{x\to 0} x^\theta \Phi(x) = 0\,, \qquad x\mapsto x^\theta \Phi(x) \;\text{ is non-decreasing,} \label{pv2}
\end{equation}
recalling that $\theta$ is defined in \eqref{p5a}. As in the previous section, we begin with an estimate on $f_n$ for small sizes, but here in $X_0$ and $X_{\Phi}$.

\begin{lemma}\label{lemv2}
For all $n\ge 1$, $T_n=\infty$ and, for any $T>0$, there is $C_{\ref{cst1}}(T)>0$ such that 
\begin{align}
M_0(f_n(t)) & \le C_{\ref{cst1}}(T)\,, \qquad t\in [0,T]\,,\ n\ge 1\,, \label{pv3} \\
M_{\Phi}(f_n(t)) & \le C_{\ref{cst1}}(T)\,, \qquad t\in [0,T]\,,\ n\ge 1\,. \label{pv4}
\end{align}
\end{lemma}

\begin{proof}
Since $b$ still satisfies \eqref{p4a}, the proof of \eqref{pv3}, along with that of $T_n=\infty$, is the same as that performed in Lemma~\ref{lemb2}
	
Next, we recall that, for $(x,y)\in (0,\infty)^2$,
\begin{equation*}
\zeta_{\Phi}(x,y) = E(x,y) \Phi(x+y) + (1-E(x,y)) \int_0^{x+y} \Phi(z) b(z,x,y)\ dz - \Phi(x) - \Phi(y)\,.
\end{equation*}

\noindent\textbf{Case~1.} If $b$ satisfies \eqref{paa4a}, then \eqref{paa4a} and the monotonicity \eqref{pv2} of $z\mapsto z^\theta \Phi(z)$ imply that
\begin{align*}
\int_0^{x+y} \Phi(z) b(z,x,y)\ dz & = \int_0^{x+y} z^\theta \Phi(z) z^{-\theta} b(z,x,y)\ dz \\
& \le (x+y)^\theta \Phi(x+y) \int_0^{x+y} z^{-\theta} b(z,x,y)\ dz \\
& \le \frac{\beta_{-\theta}'}{2} \Phi(x+y)\,.
\end{align*}
Therefore, since $\Phi$ is non-increasing,
\begin{align*}
\zeta_\Phi(x,y) & \le \left[ E(x,y) + \frac{\beta_{-\theta}'}{2} (1-E(x,y)) \right] \Phi(x+y) - \Phi(x) - \Phi(y) \\
& \le  \left[ E(x,y) + \frac{\beta_{-\theta}'}{2} (1-E(x,y)) \right] \frac{\Phi(x)+\Phi(y)}{2} - \Phi(x) - \Phi(y) \\ 
& \le \frac{\beta_{-\theta}'}{2} \left[ \Phi(x) + \Phi(y) \right]\,.
\end{align*}

\noindent\textbf{Case~2.} If $b$ satisfies \eqref{paa4}, then we infer from \eqref{paa4} and the monotonicity \eqref{pv2} of $z\mapsto z^\theta \Phi(z)$ that
\begin{align*}
\int_0^{x+y} \Phi(z) b(z,x,y)\ dz & = \int_0^x z^\theta \Phi(z) z^{-\theta} \bar{b}(z,x,y)\ dz + \int_0^y z^{\theta} \Phi(z) z^{-\theta} \bar{b}(z,y,x)\ dz \\
& \le x^\theta\Phi(x) \int_0^x z^{-\theta} \bar{b}(z,x,y)\ dz + y^\theta \Phi(y) \int_0^y z^{-\theta} \bar{b}(z,y,x)\ dz \\
& \le \frac{\beta_{-\theta}'}{2} \left[ \Phi(x) + \Phi(y) \right]\,.
\end{align*}
Consequently, using again the monotonicity of $\Phi$,
\begin{align*}
\zeta_\Phi(x,y) & \le E(x,y) \frac{\Phi(x)+\Phi(y)}{2} + \left( \frac{\beta_{-\theta}'}{2} (1-E(x,y)) - 1 \right) \left[ \Phi(x) + \Phi(y)\right] \\
& \le \frac{\beta_{-\theta}'}{2} \left[ \Phi(x) + \Phi(y) \right]\,.
\end{align*}

Summarizing, we have proved the following upper bound for $\zeta_\Phi$ in both cases:
\begin{equation}
\zeta_\Phi(x,y) \le  \frac{\beta_{-\theta}'}{2} \left[ \Phi(x) + \Phi(y) \right]\,, \qquad (x,y)\in (0,\infty)^2\,. \label{pv5}
\end{equation}
Recalling that $\alpha=0$, it then follows from \eqref{p1}, \eqref{p21}, \eqref{p9}, and the monotonicity of $\Phi$ that, for $t\in [0,T]$, 
\begin{align*}
\frac{d}{dt} M_\Phi(f_n(t)) & \le \frac{k_1 \beta_{-\theta}'}{4} \int_0^1 \int_0^1 \left[ \Phi(x) + \Phi(y) \right] f_n(t,x) f_n(t,y)\ dydx \\
& \qquad  + \frac{k_1 \beta_{-\theta}'}{4} \int_0^1 \int_1^\infty \left[ \Phi(x) + \Phi(y) \right] y f_n(t,x) f_n(t,y)\ dydx \\
& \qquad + \frac{k_1 \beta_{-\theta}'}{4} \int_1^\infty \int_0^1 \left[ \Phi(x) + \Phi(y) \right] x f_n(t,x) f_n(t,y)\ dydx \\
& \qquad + \frac{k_1 \beta_{-\theta}'}{4} \int_1^\infty \int_1^\infty \left[ \Phi(x) + \Phi(y) \right] xy f_n(t,x) f_n(t,y)\ dydx \\
& \le \frac{k_1 \beta_{-\theta}'}{2} M_0(f_n(t)) M_\Phi(f_n(t)) + \frac{\varrho k_1 \beta_{-\theta}'}{2} \left[ M_\Phi(f_n(t)) + \Phi(1) M_0(f_n(t)) \right] \\
& \qquad + \frac{k_1 \beta_{-\theta}' \Phi(1)}{2} \varrho^2\,.
\end{align*}
Hence, thanks to \eqref{pv3},
\begin{equation*}
\frac{d}{dt} M_\Phi(f_n(t)) \le C(T) \left[ 1 + M_\Phi(f_n(t)) \right]\,, \qquad t\in [0,T]\,,
\end{equation*}
and we deduce \eqref{pv4} from the previous inequality and Gronwall's lemma.
\end{proof}

\begin{proof}[Proof of Theorem~\ref{tha2b}] The proof of Theorem~\ref{tha2b} is the same as that of Theorem~\ref{tha2}, provided one replaces Lemma~\ref{lemb2} by Lemma~\ref{lemv2}. Note that, since $\Phi(x)\to\infty$ as $x\to 0$, the bound \eqref{pv4} prevents concentration of the sequence $(f_n)_{n\ge 1}$ at $x=0$ and thus plays here the same role as \eqref{p24} in the proof of Theorem~\ref{tha2}. 
\end{proof}

\subsection{Existence: $K(x,y) \le k_0(x+y)$}\label{sec2.3}

The specific behaviour of $K$ for small sizes implies an almost immediate bound on the total number of particles $M_0(f_n)$.

\begin{lemma}\label{lemb8}
For all $n\ge 1$, $T_n=\infty$ and, for any $T>0$, there is $C_{\ref{cst1}}(T)>0$ such that 
\begin{align}
M_0(f_n(t)) & \le C_{\ref{cst1}}(T)\,, \qquad t\in [0,T]\,, \ n\ge 1\,, \label{p31} \\
M_{-\theta}(f_n(t)) & \le C_{\ref{cst1}}(T)\,, \qquad t\in [0,T]\,, \ n\ge 1\,. \label{p32}
\end{align}
\end{lemma}

\begin{proof}
 By \eqref{p4a},
\begin{equation*}
\zeta_0(x,y) \le E(x,y) + \beta_0 (1-E(x,y)) - 2 = \beta_0- 2 - (\beta_0-1) E(x,y)\le \beta_0\,, \qquad (x,y)\in (0,\infty)^2\,. 
\end{equation*}
It then follows from \eqref{p7a}, \eqref{p400}, \eqref{p21}, and \eqref{p9} that, for $t\in (0,T_n)$,
\begin{align*}
\frac{d}{dt} M_0(f_n(t)) & \le \frac{k_0 \beta_0}{2} \int_0^\infty \int_0^\infty (x+y) f_n(t,x) f_n(t,y)\ dydx \\
& \le k_0 \beta_0 M_0(f_n(t)) M_1(f_n(t)) \le k_0 \beta_0 \varrho M_0(f_n(t))\,.
\end{align*}
Integrating with respect to time yields
\begin{equation*}
M_0(f_n(t)) \le M_0(f_n^{in}) e^{k_0 \beta_0 \varrho t} \le M_0(f^{in}) e^{k_0 \beta_0 \varrho t}\,, \qquad t\in [0,T_n)\,,
\end{equation*} 
which simultaneously excludes the occurrence of finite time blowup  according to \eqref{p200} and gives \eqref{p31}.

Next, since \eqref{p400} implies that $K$ satisfies \eqref{p1} with $\alpha=0$ and $k_1=2k_0$, the proof is the same as that of \eqref{p24}. 
\end{proof}

\begin{proof}[Proof of Theorem~\ref{tha4}]
The proof of Theorem~\ref{tha4} is the same as that of Theorem~\ref{tha2}~(b), the only modification being the recourse to Lemma~\ref{lemb8} instead of Lemma~\ref{lemb2}. 
\end{proof}

\section{Uniqueness}\label{sec3}

Take two solutions $f$ and $g$ to \eqref{cecb} satisfying the assumptions of Theorem~\ref{tha5} and consider $T>0$. There is $\mathcal{M}>0$ such that
\begin{equation}
M_{-2\alpha}(f(t)) + M_{-2\alpha}(g(t)) + M_2(f(t)) + M_2(g(t)) \le \mathcal{M}\,, \qquad t\in [0,T]\,. \label{p50}
\end{equation}
We next set $Z:= f-g$ and $\Sigma := \mathrm{sign}(Z)$ and define $w(x) := \max\{ x^{-\alpha},x\}$ for $x\in (0,\infty)$. For $t\in [0,T]$, it follows from \eqref{pa3} that
\begin{align*}
\frac{d}{dt} \int_0^\infty w(x) |Z(t,x)|\ dx & = \int_0^\infty w(x) \Sigma(t,x) \partial_t(f-g)(t,x)\ dx \\
& = \frac{1}{2} \int_0^\infty \int_0^\infty \zeta_{w\Sigma(t)}(x,y) K(x,y) [f(t,x) f(t,y) - g(t,x) g(t,y)]\ dydx \\
& = \frac{1}{2} \int_0^\infty \int_0^\infty \zeta_{w\Sigma(t)}(x,y) K(x,y) (f+g)(t,x) Z(t,y)\ dydx \\
& = \frac{1}{2} \int_0^\infty \int_0^\infty \zeta_{w\Sigma(t)}(x,y) K(x,y) (f+g)(t,x) \Sigma(t,y) |Z(t,y)|\ dydx\,.
\end{align*}

We now estimate $\Sigma(t,y) K(x,y) \zeta_{w\Sigma(t)}(x,y)$ according to the range of $(x,y)\in (0,\infty)^2$. To this end, we first observe that
\begin{align*}
\Sigma(t,y) \zeta_{w\Sigma(t)}(x,y) & = E(x,y) w(x+y) \Sigma(t,x+y) \Sigma(t,y) \\
& \qquad + (1-E(x,y)) \int_0^{x+y} w(z) \Sigma(t,z) \Sigma(t,y) b(z,x,y)\ dz \\
& \qquad - \Sigma(t,x) \Sigma(t,y) w(x) - w(y) \\
& \le W(x,y)\,, 
\end{align*}
with
\begin{equation*}
W(x,y) := E(x,y) w(x+y) + (1-E(x,y)) \int_0^{x+y} w(z) b(z,x,y)\ dz + w(x) - w(y)\,,
\end{equation*}
so that
\begin{equation*}
\Sigma(t,y) K(x,y) \zeta_{w\Sigma(t)}(x,y) \le K(x,y) W(x,y)\,, \qquad (x,y)\in (0,\infty)^2\,.
\end{equation*}
We also infer from \eqref{p40}, \eqref{p7a}, and \eqref{p500} that, if $x+y>1$, then
\begin{align}
W(x,y) & = E(x,y) (x+y) + (1-E(x,y)) \int_0^1 z^{-\alpha} b(z,x,y)\ dz \nonumber \\
& \qquad + (1-E(x,y)) \int_1^{x+y} z b(z,x,y)\ dz + w(x) - w(y) \nonumber \\
& \le E(x,y) (x+y) + (1-E(x,y)) B_{-\alpha} \nonumber \\
& \qquad + (1-E(x,y)) (x+y) + w(x) - w(y) \nonumber \\
& \le x+y + w(x) - w(y) + B_{-\alpha} \nonumber \\
& \le x + w(x) + B_{-\alpha}\,. \label{p51}
\end{align}

\begin{itemize}
\item If $(x,y)\in (0,1)^2$, then $w(x+y) \le 2^{1+\alpha} (x+y)^{-\alpha}$, $\min\{1, x+y\}^{-\alpha} \leq 2 (x+y)^{-\alpha}$, and it follows from \eqref{p40}, \eqref{p7a}, and \eqref{p500} that
\begin{align*}
W(x,y) & \le \frac{2^{1+\alpha}}{(x+y)^\alpha} + \int_0^{\min\{x+y,1\}} z^{-\alpha} b(z,x,y)\ dz + \mathbf{1}_{(1,\infty)}(x+y) \int_1^{x+y} z b(z,x,y)\ dz \\
& \qquad + x^{-\alpha} - y^{-\alpha} \\
& \le \frac{2^{1+\alpha}+2B_{-\alpha}}{(x+y)^\alpha} + x+y + x^{-\alpha} \le \frac{2^{2+\alpha}+2B_{-\alpha}}{(x+y)^\alpha} + x^{-\alpha} \\
& \le \left( 1+ 2^{2+\alpha}+2B_{-\alpha} \right) x^{-\alpha}\,.
\end{align*}
Together with \eqref{p1}, the above estimate gives
\begin{equation*}
K(x,y) W(x,y) \le k_1 \left( 1+ 2^{2+\alpha}+ 2B_{-\alpha} \right) x^{-2\alpha} w(y)\,, \qquad (x,y)\in (0,1)^2\,.
\end{equation*}
\item If $(x,y)\in (0,1)\times (1,\infty)$, then $x+y>1$ and it follows from \eqref{p1} and \eqref{p51} that
\begin{equation*}
K(x,y) W(x,y) \le k_1 \left( x+x^{-\alpha} + B_{-\alpha} \right) x^{-\alpha} y  \le k_1 ( 2 + B_{-\alpha}) x^{-2\alpha} w(y)\,. 
\end{equation*}
\item Similarly, if $(x,y)\in (1,\infty)\times (0,1)$, then $x+y>1$ and \eqref{p1} and \eqref{p51} ensure that
\begin{equation*}
K(x,y) W(x,y) \le k_1 \left( 2x + B_{-\alpha} \right) x y^{-\alpha}  \le k_1 ( 2 + B_{-\alpha}) x^2 w(y)\,. 
\end{equation*}
\item Finally, if $(x,y)\in (1,\infty)^2$, then $x+y>1$ and we infer from \eqref{p1} and \eqref{p51} that
\begin{equation*}
K(x,y) W(x,y) \le k_1 \left( 2x + B_{-\alpha} \right) x y \le k_1 ( 2 + B_{-\alpha}) x^2 w(y)\,. 
\end{equation*}
\end{itemize}
Summarizing, we have shown that
\begin{equation*}
\Sigma(t,y) K(x,y) \zeta_{w\Sigma(t)}(x,y)  \le k_1 \left( 1+ 2^{2+\alpha}+ 2B_{-\alpha} \right) \left( x^{-2\alpha} + x^2 \right) w(y)\,, \qquad (x,y)\in (0,\infty)^2\,. 
\end{equation*}
Therefore,
\begin{align*}
& \frac{d}{dt} \int_0^\infty w(x) |Z(t,x)|\ dx \\
& \qquad \le k_1 \left( 1+ 2^{2+\alpha}+ 2B_{-\alpha} \right) \int_0^\infty \int_0^\infty \left( x^{-2\alpha} + x^2 \right) w(y) (f+g)(t,x) |Z(t,y)|\ dydx \\
& \qquad \le k_1 \left( 1+ 2^{2+\alpha}+2B_{-\alpha} \right) \mathcal{M} \int_0^\infty  w(y) |Z(t,y)|\ dy\,,
\end{align*}
and we conclude with the help of Gronwall's lemma.

\section*{Acknowledgements}

 Part of this work was done while PhL enjoyed the hospitality of the Department of Mathematics, Indian Institute of Technology Roorkee.

 \appendix
\section{Improved integrability for small sizes}\label{sec.a}

We devote this section to a variant of the de la Vall\'ee-Poussin theorem \cite{dlVP15}, establishing an improved integrability property of integrable functions near zero.

\begin{lemma}\label{leap1}
Consider $h\in X_0$ and $\theta\in (0,1)$. There is a non-negative convex and non-increasing function $\Phi\in C^1((0,\infty))$ depending on $h$ and $\theta$ such that
\begin{equation}
\int_0^\infty \Phi(x) |h(x)|\ dx < \infty \,, \label{pap1}
\end{equation}
and
\begin{equation}
\lim_{x\to 0} \Phi(x) = \infty\,, \qquad \lim_{x\to 0} x^\theta \Phi(x) = 0\,, \qquad x \mapsto x^\theta \Phi(x) \;\text{ is non-decreasing}\,. \label{pap2}
\end{equation}
\end{lemma}

\begin{proof}
Since $h\in X_0$, the function $x\mapsto 1/x$ belongs to $L^1((0,\infty),x|h(x)| dx)$ and we infer from a refined version of the de la Vall\'ee-Poussin theorem, see \cite{Le77} and \cite[Theorem~7.1.6]{BLL19}, that there is a function $\Phi_0\in C^1([0,\infty))$ satisfying the following properties: $\Phi_0$ is convex, $\Phi_0(0)=\Phi_0'(0)=0$, $\Phi_0'$ is a concave function which is positive in $(0,\infty)$,
\begin{align}
& \lim_{\xi\to\infty} \Phi_0'(\xi) = \lim_{\xi\to\infty} \frac{\Phi_0(\xi)}{\xi} = \infty\,, \label{pap3} \\
& \lim_{\xi\to\infty} \frac{\Phi_0'(\xi)}{\xi^{p-1}} = 0 \;\text{ for }\; p\in (1,2]\,, \label{pap4}
\end{align}
and
\begin{equation}
\int_0^\infty \Phi_0\left( \frac{1}{x} \right) x |h(x)|\ dx < \infty \,. \label{pap5}
\end{equation}
Moreover, $\Phi_0$ is explicitly given by
\begin{equation*}
\Phi_0'(\xi) = \left\{ 
\begin{split}
& \frac{\xi}{j_1-j_0}\,, \qquad \xi\in [0,j_1]\,, \\
& \frac{\xi-j_m}{j_{m+1}-j_m} + m + \frac{1}{j_1-j_0}\,, \qquad \xi\in [j_m,j_{m+1}]\,, \quad m\ge 1\,,
\end{split}
\right.
\end{equation*}
and
\begin{equation*}
\Phi_0(\xi) = \int_0^\xi \Phi_0'(\xi_*)\ d\xi_*\,, \qquad \xi\ge 0\,,
\end{equation*}
where $(j_m)_{m\ge 0}$ is a sequence of positive integers which is constructed recursively and satisfies 
\begin{align*}
j_0=1\,, \quad & j_{m+1} \ge \max\{ 2j_m , e^{m+1}\}\,, \qquad m\ge 0\,, \\
& \int_0^{1/j_m} |h(x)|\ dx \le \frac{1}{m^2}\,, \qquad m\ge 1\,,
\end{align*}
see \cite[Theorem~7.1.6]{BLL19}. We next claim that
\begin{equation}
\theta \Phi_0'(\xi) - \xi \Phi_0''(\xi) \ge 2 (\theta-1)\,, \qquad \xi\in [0,\infty)\setminus \{j_m\,, \ m\ge 1\}\,. \label{pap6}
\end{equation}
Indeed, if $\xi\in [0,j_1)$, then 
\begin{equation*}
\theta \Phi_0'(\xi) - \xi \Phi_0''(\xi) = (\theta-1) \frac{\xi}{j_1-j_0} \ge (\theta-1) \frac{j_1}{j_1-j_0} \ge 2(\theta-1)\,, 
\end{equation*}
as $j_1\le 2(j_1-j_0)$. Similarly, if $m\ge 1$ and $\xi\in (j_m,j_{m+1})$, then
\begin{align*}
\theta \Phi_0'(\xi) - \xi \Phi_0''(\xi) & = (\theta-1) \frac{\xi-j_m}{j_{m+1}-j_m} + \theta m + \frac{\theta}{j_1-j_0} - \frac{j_m}{j_{m+1}-j_m} \\
& \ge \theta - 1 + \theta - 1 = 2(\theta-1)\,,
\end{align*}
as $j_m \le j_{m+1}-j_m$. We have thus proved \eqref{pap6} which gives, after integration,
\begin{equation}
(1+\theta) \Phi_0(\xi) - \xi \Phi_0'(\xi) \ge 2(\theta-1) \xi\,, \qquad \xi\ge 0\,. \label{pap7}
\end{equation}

We now set 
\begin{equation*}
\Phi(x) := x \Phi_0\left( \frac{1}{x} \right) + \frac{2}{\theta}\ge 0\,, \qquad x\in (0,\infty)\,.
\end{equation*}
Since $h\in X_0$, we first infer from \eqref{pap5} that
\begin{equation*}
\int_0^\infty \Phi(x) |h(x)|\ dx = \int_0^\infty \Phi_0\left( \frac{1}{x} \right) x |h(x)|\ dx + \frac{2}{\theta} \|h\|_{L^1(0,\infty)} < \infty\,,
\end{equation*}
hence \eqref{pap1}. It next follows from \eqref{pap3} and the convexity of $\Phi_0$ that
\begin{equation*}
\lim_{x\to 0} \Phi(x)\ge \lim_{\xi\to \infty} \frac{\Phi_0(\xi)}{\xi} = \infty\,,
\end{equation*}
and
\begin{align*}
\Phi'(x) & = \Phi_0\left( \frac{1}{x} \right) - \frac{1}{x} \Phi_0'\left( \frac{1}{x} \right) \le 0\,, \qquad x\in (0,\infty)\,, \\
\Phi''(x) & = \frac{1}{x^3} \Phi_0''\left( \frac{1}{x} \right) \ge 0\,, \qquad x\in (0,\infty)\,.
\end{align*}
Consequently, $\Phi$ is a non-negative convex and non-increasing function on $(0,\infty)$ and satisfies the first property stated in \eqref{pap2}. Moreover, by \eqref{pap4} and L'Hospital rule, 
\begin{equation*}
\lim_{x\to 0} x^\theta \Phi(x) = \lim_{x\to 0} x^{1+\theta} \Phi_0\left( \frac{1}{x} \right) = \lim_{\xi\to\infty} \frac{\Phi_0(\xi)}{\xi^{1+\theta}} = \lim_{\xi\to\infty} \frac{\Phi_0'(\xi)}{(1+\theta)\xi^\theta} = 0\,,
\end{equation*}
and we have established the second property stated in \eqref{pap2}. Finally, introducing $\Theta(x) := x^\theta \Phi(x)$, $x\in (0,\infty)$, we deduce from \eqref{pap7} that
\begin{align*}
\Theta'(x)  & = 2 x^{\theta-1} + x^\theta \left[ (1+\theta) \Phi_0\left( \frac{1}{x} \right) - \frac{1}{x} \Phi_0'\left( \frac{1}{x} \right) \right] \ge [2+2(\theta-1)] x^{\theta-1} \ge 0\,,
\end{align*}
so that $\Theta$ is non-decreasing on $(0,\infty)$.
\end{proof}

\bibliographystyle{siam}
\bibliography{CEwCB}

\end{document}